\documentclass[a4paper, 11pt,reqno]{amsart}

\usepackage[T1]{fontenc}      
\usepackage[foot]{amsaddr}    

\usepackage[british]{babel}   
\usepackage{csquotes}         
\usepackage[babel]{microtype} 

\usepackage{mathtools}        
\usepackage{amssymb}          
\usepackage{amsthm}           
\usepackage{amstext}          
\usepackage{mleftright}       
\usepackage{gensymb}          
\usepackage[makeroom]{cancel} 
\usepackage{mathdots}         
\usepackage{mathrsfs}         
\usepackage{dsfont}           

\usepackage{stickstootext}
\usepackage[stickstoo,varbb]{newtxmath} 
\makeatletter
\DeclareFontFamily{U}{ntxmia}{\skewchar \font =127}
 \DeclareFontShape{U}{ntxmia}{m}{it}{
                        <-> \ntxmath@scaled ntxmia
                      }{}    
                      \DeclareFontShape{U}{ntxmia}{b}{it}{
                        <-> \ntxmath@scaled ntxbmia
                      }{}
\makeatother

\usepackage{graphicx}         
\usepackage{psfrag}           
\usepackage{caption}          
\usepackage[labelformat=simple]{subcaption}

\usepackage{tikz}             
\usetikzlibrary{backgrounds,shapes.geometric,positioning,calc}

\usepackage{booktabs}         

\usepackage{enumitem}         

\usepackage[margin=1in]{geometry} 
\usepackage[indent=12pt]{parskip}              
\newlength\docparskip
\usepackage[textsize=footnotesize]{todonotes}

\usepackage[numbers,sort]{natbib}
\makeatletter
\def\NAT@spacechar{~}
\makeatother

\usepackage{hyperref}              
\usepackage[capitalise]{cleveref}  
\hypersetup{colorlinks=true,
   citecolor=blue,
   filecolor=blue,
   linkcolor=blue,
   urlcolor=blue
}
\usepackage{url}

\makeatletter
\if@cref@capitalise
\crefname{figure}{Figure}{Figures}
\crefname{claim}{Claim}{Claims}
\else
\crefname{figure}{Figure}{Figures}
\crefname{claim}{claim}{claims}
\fi
\makeatother
\Crefname{figure}{Figure}{Figures}
\Crefname{claim}{Claim}{Claims}

\Crefname{enumi}{}{}
\usepackage{algorithm}
\usepackage{algpseudocode}

\allowdisplaybreaks

\theoremstyle{definition}
\newtheorem{definition}{Definition}

\theoremstyle{plain}
\newtheorem{claim}{Claim}

\newtheorem{theorem}[definition]{Theorem}
\newtheorem{corollary}[definition]{Corollary}
\newtheorem{lemma}[definition]{Lemma}

\newtheorem{problem}[definition]{Problem}

\newenvironment{claimproof}{%
\let\origqed=\qedsymbol%
\renewcommand{\qedsymbol}{$\blacktriangleleft$}%
\begin{proof}}{\end{proof}\let\qedsymbol=\origqed}

\renewcommand{\binom}[2]{\ensuremath{\mleft(\kern-.1em\genfrac{}{}{0pt}{}{#1}{#2}\kern-.1em\mright)}}    
\newcommand{\inbinom}[2]{\ensuremath{\bigl(\kern-.1em\genfrac{}{}{0pt}{}{#1}{#2}\kern-.1em\bigr)}} 

\newcommand*\nume{\ensuremath{\mathrm{e}}}

\newcommand{\cC}{\mathscr{C}}

\makeatletter
\def\moverlay{\mathpalette\mov@rlay}
\def\mov@rlay#1#2{\leavevmode\vtop{%
  \baselineskip\z@skip \lineskiplimit-\maxdimen
  \ialign{\hfil$\m@th#1##$\hfil\cr#2\crcr}}}
\newcommand{\charfusion}[3][\mathord]{
    #1{\ifx#1\mathop\vphantom{#2}\fi
        \mathpalette\mov@rlay{#2\cr#3}
      }
    \ifx#1\mathop\expandafter\displaylimits\fi}
\makeatother

\newcommand{\eps}{\epsilon}

\newcommand{\COMMENT}[1]{}
\newcommand{\COMNEW}[1]{}
\renewcommand{\COMNEW}[1]{\footnote{\textcolor{red!70!black}{#1}}} 

\title{On $2$-factors of Hamiltonian graphs}

\author[A.~Espuny D\'iaz]{Alberto Espuny D\'iaz}
\email{aespuny@ub.edu}
\address[Espuny D\'iaz]{Departament de Matem\`atiques i Inform\`atica, Universitat de Barcelona (UB), Gran Via de les Corts Catalanes, 585, 08007 Barcelona, Spain.}
\author[A.~Gir\~ao]{Ant\'onio Gir\~ao}
\email{a.girao@ucl.ac.uk}
\address[Gir\~ao]{Department of Mathematics, University College London, London, WC1E\,6BT, United Kingdom.}
\author[B.~Granet]{Bertille Granet}
\email{bertille.granet@warwick.ac.uk}
\address[Granet]{Warwick Mathematics Institute, University of Warwick, Coventry CV4\,7AL, United Kingdom.}
\author[G.~Kronenberg]{Gal Kronenberg}
\email{kronenberg@maths.ox.ac.uk}
\address[Kronenberg]{Mathematical Institute, University of Oxford, Andrew Wiles Building, Radcliffe Observatory Quarter, Woodstock Road, Oxford, OX2\,6GG, United Kingdom.}

\thanks{Alberto Espuny D\'iaz was partially funded by the Deutsche Forschungsgemeinschaft (DFG, German Research Foundation) through project no.\ 513704762.
Bertille Granet was supported by the Australian Research Council grant DP240101048.
Gal Kronenberg is supported by the Royal Commission for the Exhibition of 1851.}

\date{}

\begin{document}

\begin{abstract}
    Let $k\geq 2$.
    We show that, for a sufficiently small $\eps>0$, any sufficiently large $n$-vertex Hamiltonian graph of minimum degree at least $n^{1-\eps}$ contains a $2$-factor consisting of exactly $k$ cycles.
    This is the first minimum-degree condition which is polynomially smaller than linear.
    Our methods yield an analogous result when the host graph is not required to contain a Hamilton cycle, but only a $2$-factor consisting of at most $k$ cycles; this answers a question of Buci\'c, Jahn, Pokrovskiy and Sudakov.
\end{abstract}

\maketitle

\section{Introduction}

A \emph{Hamilton cycle} in a graph $G$ is a cycle that goes through every vertex of~$G$, and a graph is \emph{Hamiltonian} whenever it contains a Hamilton cycle.
The problem of determining whether a graph is Hamiltonian is famously NP-complete~\cite{Karp}, and thus much effort has been devoted to finding simple sufficient conditions for Hamiltonicity.
Perhaps the most famous of these is Dirac's condition:
his classical theorem from 1952~\cite{Dirac52} states that every graph $G$ on $n\geq3$ vertices with minimum degree $\delta(G)\geq n/2$ contains a Hamilton cycle, and this minimum-degree condition is best possible, as can be seen by considering a slightly unbalanced complete bipartite graph.

A \emph{$2$-factor} of a graph $G$ is a spanning $2$-regular subgraph.\COMMENT{The problem of whether a graph contains a $2$-factor or not is in P. In fact, a ``simple'' characterisation was first given by \citet{Belck50} (see the work of \citet{Tutte52} for a much more general result).}
With this language, Dirac's theorem provides an optimal minimum-degree condition for a graph to have a $2$-factor consisting of only one cycle.
It is natural to wonder how the problem changes if, rather than exactly one cycle, we wish to find a $2$-factor consisting of exactly $k$ cycles.
This problem remains NP-complete~\cite[p.~75]{GJ79}, and hence, again, it is natural to search for sufficient conditions for this property.
In 1997, \citet{BCFGL97} extended Dirac's theorem, showing that, for any positive integer~$k$ and provided $n\geq4k-1$ (cf.~\cite{Shuya18}), Dirac's condition in fact ensures the existence of a $2$-factor consisting of exactly $k$ cycles.\COMMENT{In fact, they obtain a stronger Ore-type degree condition.}\COMMENT{\Citet{AF96} had earlier proved a result \cite[Proposition~5.1]{AF96} which implies an approximate version of this result.}\COMMENT{There are multiple other conditions which imply that a graph has a $2$-factor with exactly $k$ cycles; see~\cite{CY18} for a survey.}\COMMENT{In a closely related direction, \citet{ElZ84} considered optimal degree conditions under which one can find not only a $2$-factor with a desired number of cycles, but also prescribe the lengths of these cycles.
He proved this conjecture in the case $k=2$, and subsequently \citet{Abbasi98} and \citet{KSS11} resolved it in general for sufficiently large~$n$. There are several papers considering particular cases of this conjecture.}
(For several other sufficient conditions for this property, see, e.g.,~\cite{AB93,CFGS00,FGJLS04,Shuya18,CJY20,ZY22,CY26}).\COMMENT{Some more references: the results in \cite{CGKOSS07} are superseded by those in \cite{Shuya18}; \cite{Wang99,LWY01} are for bipartite graphs; \cite{GH99,GH06} are about conditions for line graphs.}
This minimum-degree condition is again best possible.
However, all graphs showing that this condition is tight are the same as for Dirac's theorem; in particular, they are not Hamiltonian.
This leads to the question of whether assuming that $G$ is Hamiltonian could significantly reduce the sufficient minimum-degree condition for having a \mbox{$2$-factor} consisting of exactly $k\geq2$ cycles.

In 2005, \citet{FGJLS05} considered this question and conjectured an affirmative answer, which they proved for the particular case that $k=2$, showing that minimum degree $\delta(G)\geq5n/12+2$ suffices for all $n\geq6$.
Soon after, \citet{Sar08} proved the conjecture for all (fixed)~$k$, reducing the lower bound of $n/2$ on the minimum degree by showing that there exists some $\eps>0$ such that, for every $k\geq2$ and $n$ sufficiently large, having minimum degree $\delta(G)\geq(1/2-\eps)n$ suffices for an $n$-vertex Hamiltonian graph~$G$ to contain a $2$-factor consisting of~$k$ cycles.
\citet{DFM14} subsequently improved on the constant in the leading term of the minimum-degree condition, showing that, for every fixed $\eps>0$ and  sufficiently large~$n$, having minimum degree $\delta(G)\geq(2/5+\eps)n$ suffices.
More recently, \citet{BJPS20} substantially improved the previous results by showing that sublinear degrees suffice: for any $k\geq2$ and $\eps>0$, there exists an $n_0\in\mathbb{N}$ such that every Hamiltonian graph on $n\geq n_0$ vertices with $\delta(G)\geq\eps n$ contains a $2$-factor consisting of exactly~$k$ cycles.

Here we make new progress on this problem by providing the first polynomial-factor improvement on the sufficient minimum-degree condition for the existence of a $2$-factor consisting of exactly $k$ cycles in a Hamiltonian graph.
We also generalise previous results in two ways.
First, we allow the number of cycles $k$ to grow polynomially in~$n$.
Second, we do not in fact require $G$ to contain a Hamilton cycle, but only a $2$-factor consisting of at most $k$ cycles.
This second generalisation answers a question of \citet{BJPS20}.

\begin{theorem}\label{thm:main}
    For all $0< \delta \leq 1$, there exist $\eps>0$ and $n_0\in \mathbb{N}$ such that the following holds.
    Let~$G$ be a graph on $n\geq n_0$ vertices with minimum degree $\delta(G)\geq n^{1-\eps}$ and let $k\in \{1,\ldots,\lfloor n^{1-\delta}\rfloor\}$.
    If~$G$ has a $2$-factor consisting of at most~$k$ cycles, then it contains a $2$-factor consisting of exactly~$k$ cycles.
\end{theorem}

\Citet{BJPS20} noted that, in fact, a careful analysis of their proof yields an (explicit) expression of the form $\delta(G)\geq n/{\operatorname{poly}\log\log n}$ as a sufficient minimum-degree condition for having a $2$-factor consisting of exactly~$k$ cycles (for any fixed $k$ and sufficiently large~$n$).
However, they point out that their methods do not seem sufficient for improving the minimum-degree condition by a polynomial factor, like we achieve in \cref{thm:main}.
Indeed, our approach for obtaining this improvement is very different from that in their paper.

\section{Notation}

Given any $x\in\mathbb{R}$, we will denote $[x]\coloneqq\{k\in\mathbb{Z}:1\leq k\leq x\}$.
To simplify the presentation, we will often write \emph{hierarchies} to express the relationships between parameters in our statements.
Formally, if we write that a statement holds for $0<a\ll b\leq 1$, we mean that for all $b\in(0,1]$ there exists some~$a_0>0$ such that, for all $a\in(0,a_0]$, the statement holds for this choice of~$a$ and~$b$.
This notation extends in the natural way to hierarchies which are longer or where one parameter depends on several others, with parameters being defined from right to left.
When one of the parameters in a hierarchy is written as~$1/k$, the reader should interpret that~$k\in\mathbb{N}$.
For the sake of streamlining the presentation, we usually ignore rounding when this does not affect the argument.

Some of our statements will be asymptotic in nature.
When working with asymptotic statements in probabilistic settings, we will say that an event occurs \emph{asymptotically almost surely} (a.a.s.) if the probability that it holds tends to $1$ as $n$ tends to infinity.

Our graph-theoretic notation is standard.
Let $G$ be a graph with vertex set~$V(G)$ and edge set~$E(G)$, and let $E\subseteq E(G)$ be a set of edges and $S,T\subseteq V(G)$ be disjoint sets of vertices.
We write \mbox{$V(E)\coloneqq\bigcup_{e\in E}e$}.
We denote by $G\setminus E$ the graph on vertex set $V(G)$ with edge set $E(G)\setminus E$.
The subgraph of~$G$ \emph{induced} by~$S$, that is, the graph with vertex set $S$ and edge set $\{e\in E(G):e\subseteq S\}$, is denoted by~$G[S]$.
We set $G-S\coloneqq G[V(G)\setminus S]$.
We write $e(G)\coloneqq|E(G)|$.
The set of all edges $e\in E(G)$ which join a vertex in~$S$ to one in~$T$ is denoted by $E_G(S,T)$, and we set $e_G(S,T)\coloneqq|E_G(S,T)|$.
If $S=\{v\}$, we simplify the notation to $e_G(v,T)$ (and similarly for the rest of the notation); we also write a simple~$uv$ to denote an edge $\{u,v\}\in E(G)$.
The \emph{neighbourhood} in $G$ of a vertex $v\in V(G)$ is the set $N_G(v)\coloneqq\{u\in V(G):uv\in E(G)\}$.
We denote $N_G(S)\coloneqq\bigcup_{v\in S}N_G(v)$.
The \emph{degree} of~$v$ in~$G$ is $d_G(v)\coloneqq|N_G(v)|$, and the minimum of the degrees in~$G$ of all vertices of~$G$ is denoted by~$\delta(G)$.

We denote a cycle of length $4$ as $C_4$.
We say two non-incident edges $xy,zw$ \emph{form} a $C_4$ in a graph $G$ if $xy,yz,zw,wx\in E(G)$ or $xy,yw,wz,zx\in E(G)$. 
Given a graph $G$ which contains some (longer) cycle $\cC$, we say that a copy $C$ of $C_4$ in $G$ is \emph{implanted} in $\cC$ if two non-incident edges of~$C$ are in $\cC$ and the other two edges are not (but they are still in the graph $G$).
Given a cycle $\cC$ with an implanted cycle~$C$ of length $4$, a \emph{$C_4$-switch} is the operation of deleting the two edges of $C$ which are in $\cC$ and adding the two which are not in $\cC$ (but are in $G$), that is, the result of the $C_4$-switch is $\cC\triangle C$. 

\section{Proof overview}\label{sec:sketch}

A general strategy for constructing a $2$-factor consisting of exactly $k$ cycles is to begin with an arbitrary $2$-factor, typically one with fewer than $k$ components, and then refine it by partitioning some longer cycles into shorter ones.

A natural mechanism for splitting a long cycle is to identify a suitable collection of chords and use them to induce a partition.
One instance of such a configuration is given by the set of chords which are edges of a cycle of even length $\ell$ which alternates between edges of the long cycle and chords.
Depending on how these chords lie with respect to each other, the resulting $2$-factor will have a different number of components (in particular, it might even be another Hamilton cycle).
In order to control the number of cycles, the method introduced by \citet{BJPS20} relies on imposing some specific pattern on the chords and employing blow-ups of these structures to allow for additional fine-tuning.
However, this technique entails an inherent limitation with respect to lowering the minimum degree condition, namely the existence of a suitably large blow-up.
Our approach departs from this framework by exploiting a family of structures arising from the presence of many copies of~$C_4$ in the host graph.

The basic idea is to show that, if a long cycle ``contains'' sufficiently many appropriate copies of~$C_4$, then it can be partitioned into multiple shorter cycles (see \cref{sect:C4switch}).
Towards our goal of partitioning a long cycle into several cycles, it is essential to analyse how distinct copies of $C_4$ interact.
We therefore employ results concerning graph drawings to ensure the existence of sufficiently many copies of $C_4$ within each long cycle that interact in a structurally favourable manner. 
These results are stated in \cref{sec:drawings} and applied in \cref{sec:switches}. 

However, although the prescribed minimum degree condition ensures an abundance of~$C_4$'s in the graph as a whole, these copies may not lie along the given long cycle in the required configuration.
Consequently, it will be crucial for us to transform the original long cycle into another long cycle that exhibits the desired structural property.
Our main novel idea is a two-step process for ``forcing'' many $C_4$'s into a Hamilton cycle (see \cref{sect:sketch2}).
To this end, we invoke a structural result of Thomassen (see \cref{lem:thomassen}), which permits controlled modifications of the edge set of a long cycle while preserving a specified collection of ``protected'' edges.

We now explain the key proof ingredients in more detail.
Let $0< 1/n\ll \eps \ll \eta\ll \delta \leq 1$.
For simplicity, we assume here that $\ell=1$.
Let $G$ be an \mbox{$n$-vertex} graph with minimum degree at least $n^{1-\eps}$ and containing a Hamilton cycle~$\cC$.
We want to transform~$\cC$ into a $2$-factor consisting of a desired number $k$ of cycles.

\subsection{\texorpdfstring{$C_4$}{C4}-switches}\label{sect:C4switch}

As seen in \cref{fig:non}, if $C=x_1x_2x_3x_4$ is a copy of~$C_4$ which is implanted in~$\cC$ and whose edges in $G\setminus \cC$ are ``non-crossing'', then $\cC \triangle C$ is a $2$-factor which consists of precisely two cycles.
Here the fact that the edges in $G\setminus \cC$ are ``non-crossing'' is crucial.
Indeed, as seen in \cref{fig:cross}, if they are ``crossing'', then $\cC\triangle C$ is another Hamilton cycle.
However, if we can find two ``crossing'' $C_4$'s $C=x_1x_2x_3x_4$ and $C'=y_1y_2y_3y_4$ which are ``crossing'' each other as in \cref{fig:two}, then $\cC\triangle (C \cup C')$ is again a $2$-factor consisting of precisely two cycles.

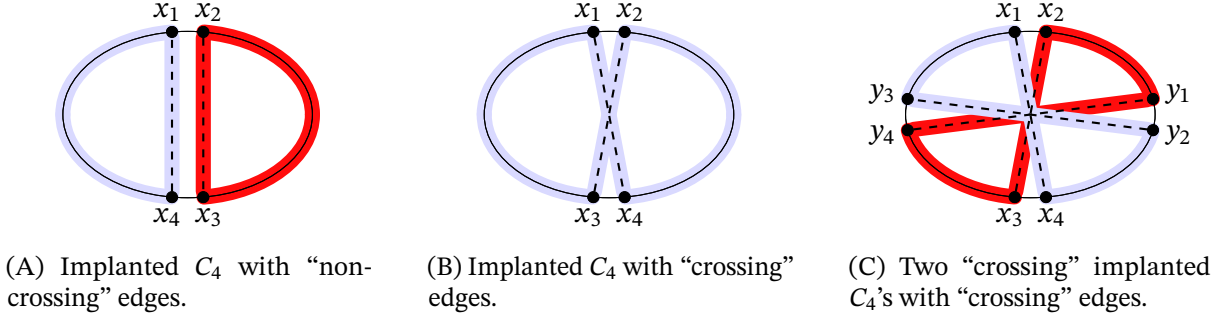
\begin{figure}[h]
    \centering

\begin{subfigure}{0.3\textwidth}
\centering            
\begin{tikzpicture}[scale=1.1]
\draw[line join=round,line cap=round,line width=.2cm, red!95] (1.5,0) arc(0:-82.5:1.5cm and 1cm);
\draw[line join=round,line cap=round,line width=.2cm, red!95] (1.5,0) arc(0:82.5:1.5cm and 1cm);
\draw[line join=round,line cap=round,line width=.2cm, blue!15] (-1.5,0) arc(180:97.5:1.5cm and 1cm);
\draw[line join=round,line cap=round,line width=.2cm, blue!15] (-1.5,0) arc(180:262.5:1.5cm and 1cm);
\draw (1.5,0) arc [start angle=0,end angle=360,x radius=1.5,y radius=1]
        node [pos=.73,draw,ellipse, scale=.4, fill] (x1) {}
        node [pos=.77,draw,ellipse, scale=.4, fill] (x2) {}
        node [pos=.23,draw,ellipse, scale=.4, fill] (x3) {}
        node [pos=.27,draw,ellipse, scale=.4, fill] (x4) {};
\draw[line join=round,line cap=round,line width=.2cm, blue!15]  (x1) to  (x4);
\draw[line join=round,line cap=round,line width=.2cm, red!95]  (x3) to  (x2);

\draw (1.5,0) arc [start angle=0,end angle=360,x radius=1.5,y radius=1]
        node [pos=.73,draw,ellipse, scale=.4, fill] (x1) {}
        node [pos=.77,draw,ellipse, scale=.4, fill] (x2) {}
        node [pos=.23,draw,ellipse, scale=.4, fill] (x3) {}
        node [pos=.27,draw,ellipse, scale=.4, fill] (x4) {};
\draw[dashed, thick]  (x1) to  (x4);
\draw[dashed, thick]  (x3) to  (x2);
\filldraw[black] (-.25,1) circle (0pt) node [anchor=south]{$x_1$};
\filldraw[black] (.25,1) circle (0pt) node[anchor=south]{$x_2$};
\filldraw[black] (-.25,-1) circle (0pt) node[anchor=north]{$x_4$};
\filldraw[black] (.25,-1) circle (0pt) node[anchor=north]{$x_3$};
\end{tikzpicture}
\caption{Implanted $C_4$ with ``non-crossing'' edges.}
\label{fig:non}
\end{subfigure}
\hfill
\begin{subfigure}{0.3\textwidth}
\centering
\begin{tikzpicture}[scale=1.1]
\draw[line join=round,line cap=round,line width=.2cm, blue!15] (1.5,0) arc(0:-82.5:1.5cm and 1cm);
\draw[line join=round,line cap=round,line width=.2cm, blue!15] (1.5,0) arc(0:82.5:1.5cm and 1cm);
\draw[line join=round,line cap=round,line width=.2cm, blue!15] (-1.5,0) arc(180:97.5:1.5cm and 1cm);
\draw[line join=round,line cap=round,line width=.2cm, blue!15] (-1.5,0) arc(180:262.5:1.5cm and 1cm);
\draw (1.5,0) arc [start angle=0,end angle=360,x radius=1.5,y radius=1]
        node [pos=.73,draw,ellipse, scale=.4, fill] (x1) {}
        node [pos=.77,draw,ellipse, scale=.4, fill] (x2) {}
        node [pos=.23,draw,ellipse, scale=.4, fill] (x3) {}
        node [pos=.27,draw,ellipse, scale=.4, fill] (x4) {};
\draw[line join=round,line cap=round,line width=.2cm, blue!15]  (x1) to  (x3);
\draw[line join=round,line cap=round,line width=.2cm, blue!15]  (x4) to  (x2);

\draw (1.5,0) arc [start angle=0,end angle=360,x radius=1.5,y radius=1]
        node [pos=.73,draw,ellipse, scale=.4, fill] (x1) {}
        node [pos=.77,draw,ellipse, scale=.4, fill] (x2) {}
        node [pos=.23,draw,ellipse, scale=.4, fill] (x3) {}
        node [pos=.27,draw,ellipse, scale=.4, fill] (x4) {};
\draw[dashed, thick]  (x1) to  (x3);
\draw[dashed, thick]  (x4) to  (x2);
\filldraw[black] (-.25,1) circle (0pt) node [anchor=south]{$x_1$};
\filldraw[black] (.25,1) circle (0pt) node[anchor=south]{$x_2$};
\filldraw[black] (-.25,-1) circle (0pt) node[anchor=north]{$x_3$};
\filldraw[black] (.25,-1) circle (0pt) node[anchor=north]{$x_4$};
\end{tikzpicture}
\caption{Implanted $C_4$ with ``crossing'' edges.}
\label{fig:cross}
\end{subfigure}
\hfill
\begin{subfigure}{0.3\textwidth}
\centering
\begin{tikzpicture}[scale=1.1]
\draw[line join=round,line cap=round,line width=.2cm, red!95] (.19,.99) arc(82.5:11:1.5cm and 1cm);
\draw[line join=round,line cap=round,line width=.2cm, blue!15] (.19,-.99) arc(-82.5:-11:1.5cm and 1cm);
\draw[line join=round,line cap=round,line width=.2cm, blue!15] (-.19,.99) arc(97.5:169:1.5cm and 1cm);
\draw[line join=round,line cap=round,line width=.2cm, red!95] (-.19,-.99) arc(262.5:191:1.5cm and 1cm);
\draw (1.5,0) arc [start angle=0,end angle=360,x radius=1.5,y radius=1]
        node [pos=-.03,draw,ellipse, scale=.4, fill] (y1) {}
        node [pos=.03,draw,ellipse, scale=.4, fill] (y2) {}
        node [pos=.47,draw,ellipse, scale=.4, fill] (y3) {}
        node [pos=.53,draw,ellipse, scale=.4, fill] (y4) {}
        node [pos=.73,draw,ellipse, scale=.4, fill] (x1) {}
        node [pos=.77,draw,ellipse, scale=.4, fill] (x2) {}
        node [pos=.23,draw,ellipse, scale=.4, fill] (x3) {}
        node [pos=.27,draw,ellipse, scale=.4, fill] (x4) {};
\draw[line join=round,line cap=round,line width=.2cm, red!95]  (x1) to  (x3);
\draw[line join=round,line cap=round,line width=.2cm, red!95]  (y4) to  (y2);
\draw[line join=round,line cap=round,line width=.2cm, blue!15]  (x4) to  (x2);
\draw[line join=round,line cap=round,line width=.2cm, blue!15]  (y1) to  (y3);
\draw (1.5,0) arc [start angle=0,end angle=360,x radius=1.5,y radius=1]
        node [pos=-.03,draw,ellipse, scale=.4, fill] (y1) {}
        node [pos=.03,draw,ellipse, scale=.4, fill] (y2) {}
        node [pos=.47,draw,ellipse, scale=.4, fill] (y3) {}
        node [pos=.53,draw,ellipse, scale=.4, fill] (y4) {}
        node [pos=.73,draw,ellipse, scale=.4, fill] (x1) {}
        node [pos=.77,draw,ellipse, scale=.4, fill] (x2) {}
        node [pos=.23,draw,ellipse, scale=.4, fill] (x3) {}
        node [pos=.27,draw,ellipse, scale=.4, fill] (x4) {};
\draw[dashed, thick]  (x1) to  (x3);
\draw[dashed, thick]  (x4) to  (x2);
\draw[dashed, thick]  (y1) to  (y3);
\draw[dashed, thick]  (y4) to  (y2);
\filldraw[black] (-.25,1) circle (0pt) node [anchor=south]{$x_1$};
\filldraw[black] (.25,1) circle (0pt) node[anchor=south]{$x_2$};
\filldraw[black] (-.25,-1) circle (0pt) node[anchor=north]{$x_3$};
\filldraw[black] (.25,-1) circle (0pt) node[anchor=north]{$x_4$};
\filldraw[black] (1.5,.25) circle (0pt) node [anchor=west]{$y_1$};
\filldraw[black] (1.5,-.25) circle (0pt) node[anchor=west]{$y_2$};
\filldraw[black] (-1.5,-.25) circle (0pt) node[anchor=east]{$y_4$};
\filldraw[black] (-1.5,.25) circle (0pt) node[anchor=east]{$y_3$};
\end{tikzpicture}
\caption{Two ``crossing'' implanted $C_4$'s with ``crossing'' edges.}
\label{fig:two}
\end{subfigure}

    \caption{A Hamilton cycle (full black) with one or two implanted $C_4$'s (dashed black). Performing $C_4$ switches gives rise to a new $2$-factor (highlighted) consisting of one or two cycles depending on the configuration.}
    \label{fig:sketch}
\end{figure}

More generally, we will see in \cref{lm:switch} that, in order to transform a Hamilton cycle into a $2$-factor consisting of exactly $k$ cycles, it suffices to find a suitable combination of at most $3k$~$C_4$'s which are implanted in $\cC$.
Using results about drawings of graphs on the plane, we will show that we can find such a suitable combination as long as the total number of $C_4$'s implanted in~$\cC$ is large (say, at least $n^{2-10\eta}$).
It thus suffices to prove that there is a Hamilton cycle with that many implanted~$C_4$'s.

\subsection{Finding a Hamilton cycle with many implanted \texorpdfstring{$C_4$}{C4}'s}\label{sect:sketch2}

Our key tool will be a result which follows from the work of \citet{Tho97} (see \cref{lem:thomassen_new} for a precise statement).
Suppose~$\cC_0$ is a Hamilton cycle of $G$ and that $E$ is a small set of good edges in $\cC_0$ which we want to preserve.
Suppose, moreover, that there is a graph $G'\subseteq G$ which consists of desirable edges which we would like to incorporate into~$\cC_0$.
If all but a small set $B$ of bad vertices have large degree in~$G'$, then \cref{lem:thomassen_new} guarantees that $G$ has a Hamilton cycle $\cC_0'$ which contains all the edges in $E$ as well as an edge from~$G'$.

Suppose for simplicity that any two vertices have many common neighbours, which implies that any two vertices are contained in many $C_4$'s.
More precisely, we have the following property: if $uv\in E(G)$, then for any neighbour $w\neq u$ of $v$ there are many edges $wx$ which form a~$C_4$ with~$uv$.
We say a set $E$ of edges is \emph{good} if it is small and almost all vertices $v\in V(G)$ are incident to many edges $uv$ which form a $C_4$ with an edge in $E$.

First, we prove that $G$ has a Hamilton cycle with at least $n^{1-3\eta}$ disjoint good sets.
Suppose not, and let $\cC_1$ be a Hamilton cycle with the maximum number of disjoint good sets $E_1,\dots, E_t$ and which contains one more small set $E_{t+1}$ which is as close to being good as possible.
We want to use \cref{lem:thomassen_new} to contradict the maximality of $\cC_1$.
The union $E_1\cup \dots \cup E_t\cup E_{t+1}$ will be the small set $E$ that we preserve.
We let $G'$ consist of all the edges $e$ of $G\setminus E(\cC_1)$ such that $E_{t+1}\cup \{e\}$ is significantly closer to being good than $E_{t+1}$.
Using that any two vertices have many common neighbours, one can show that~$\delta(G')$ is large.
Thus, \cref{lem:thomassen_new} implies that $G$ has a Hamilton cycle containing $E_1\cup\dots\cup E_t\cup E_{t+1}$ as well as one of these edges $e$, a contradiction.

We will now only consider Hamilton cycles which contain $t\coloneqq n^{1-3\eta}$ fixed good sets $E_1, \dots, E_t$.
We say an edge $e$ is compatible with $E_i$ if $e$ forms a $C_4$ with one of the edges in~$E_i$.
An edge is \emph{good} if it is compatible with $n^{1-4\eta}$ good sets $E_i$.
We let $\cC_2$ be a Hamilton cycle containing $E_1, \dots, E_t$ as well as a maximum number of good edges.
If there are at least $n^{1-6\eta}$ good edges, then the number of implanted $C_4$'s is at least $n^{2-10\eta}$, and we are done.
So suppose not.
Once again, we want to use \cref{lem:thomassen_new} to find a contradiction, this time with the fact that $\cC_2$ contains the maximum number of good edges.
We want to preserve all the edges in $E_1\cup \dots \cup E_t$ as well as all good edges in $\cC_2$.
Now we let $G'$ be the set of all good edges in $G\setminus E(\cC_2)$.
Using that any two vertices have many common neighbours, one can show that all but a small set $B$ of bad vertices have large degree in $G'$.
Hence, by \cref{lem:thomassen_new}, $G$ has a Hamilton cycle $\cC_2'$ containing $E_1,\dots, E_t$ plus at least one more good edge, a contradiction.

In general, however, it is not true that any two vertices in $G$ have many common neighbours.
Nevertheless, we can partition~$V(G)$ into few parts such that this property holds for any two vertices in a common part (see \cref{lm:part}).
So instead we essentially apply the above argument to each part in parallel.
In practice, this means we will have to deal with both arguments which make use of \cref{lem:thomassen_new} at once (see \cref{cor:stage}).

\section{Preliminaries}

\subsection{Chernoff's bound}

We will use the following standard Chernoff bound (see, e.g., the book of \citet[Corollary~2.3 and Theorem~2.10]{JLR}).

\begin{lemma}\label{lem:Chernoff}
Let\/ $X$ be a sum of independent Bernoulli random variables, or a hypergeometric random variable.
Then, for all\/ $0<\delta<1$ we have that\/
$\mathbb{P}[|X-\mathbb{E}[X]|\geq\delta \mathbb{E}[X]]\leq2\nume^{-\delta^2\mathbb{E}[X]/3}$.
\end{lemma}

\subsection{Thomassen's lemma}\label{sec:Thomassen}

The following classical result of \citet{Tho97} (which uses Thomason's lollipop method~\cite{Tho78}) provides conditions which guarantee the existence of at least two different Hamilton cycles in a Hamiltonian graph.

Given a graph $G$ and a vertex set $S\subseteq V(G)$, we say that $S$ is \emph{$G$-independent} if $G[S]$ is empty.
Given a set $T\subseteq V(G)\setminus S$, we say that $S$ \emph{dominates} $T$ in $G$ if for every $x\in T$ we have that $e_G(x,S)>0$.

\begin{lemma}[{\citet[Theorem~2.2]{Tho97}}]\label{lem:thomassen}
    Let $G$ be a graph with a Hamilton cycle $\cC$.
    Let $S\subseteq V(G)$ be a $\cC$-independent set which dominates $N_{\cC}(S)$ in $G\setminus \cC$.
    Then, $G$ contains a Hamilton cycle~$\cC'$ different from~$\cC$ and such that $\cC-S=\cC'-S$.
\end{lemma}

We note that, in the original statement in \cite{Tho97}, the set $S$ actually dominates the entire vertex set.
This is however not necessary, since the final condition $\cC-S=\cC'-S$ implies that the edges of~$G\setminus \cC$ between $S$ and $V(G)\setminus (S\cup N_{\cC}(S))$ cannot be used to construct the new Hamilton cycle $\cC'$.

Using this, we can prove the following result, which we will use to show that, if we have a Hamilton cycle~$\cC$ which contains a small set $E\subseteq E(\cC)$ of ``good'' edges which we want to preserve and a graph~$G$ consisting of ``desirable'' edges we would like to incorporate is sufficiently dense, then there exists a Hamilton cycle $\cC'$ which still contains all the edges in $E$ as well as at least one new edge from~$G$.
The proof is straightforward: we simply show that, by choosing a random set of vertices, with high probability this set will satisfy the properties required for it to play the role of $S$ in \cref{lem:thomassen}.

\begin{corollary}\label{lem:thomassen_new}
    Let $0<1/n\ll 1$.
    Let $G$ be an $n$-vertex graph and $\cC$ a Hamilton cycle on $V(G)$.
    Let $E\subseteq E(\cC)$ and $B\subseteq V(G)$ be such that $B'\coloneqq B\cup V(E)\neq V(G)$.\COMMENT{This condition is needed so that the next condition is not vacuously true, in which case the statement might not hold.}
    Suppose that for all $v\in V(G)\setminus B'$ we have that $d_{G}(v)\geq\sqrt{n}\log^2n+3|B'|+2$.
    Then $G\cup \cC$ contains a Hamilton cycle~$\cC'$ different from~$\cC$ such that $E\subseteq E(\cC')$.
\end{corollary}

\begin{proof}
    Let $A\coloneqq V(G)\setminus(N_{\cC}(B')\cup B')$ and $G'\coloneqq G\setminus\cC$, and note that $|N_{G'}(v)\cap A|\geq\sqrt{n}\log^2n$ for all $v\in V(G)\setminus B'$.
    Now let $S\subseteq A$ be sampled by including each vertex of $A$ independently with probability $1/(n\log n)^{1/2}$.
    We shall show that a.a.s.\ $S$ satisfies the properties required in \cref{lem:thomassen}.

    First, note that, for each vertex $v\in A$,\COMMENT{At least two of the three vertices we consider must be in $S$, there being two possible ways for this to happen ($v$ itself must be in $S$), and each is added to $S$ independently.} 
    \[\mathbb{P}[\{v\in S\}\wedge \{S\cap N_\cC(v)\neq\varnothing\}]\leq\frac{2}{n\log n}.\]
    It follows immediately by a union bound that a.a.s.\ $S$ is $\cC$-independent.

    Second, for any fixed vertex $v\in V(G)\setminus B'$, we have that
    \COMMENT{For the last inequality we use that $1-x\leq e^{-x}$ for any $x\in \mathbb{R}$.}
    \[\mathbb{P}[S\cap N_{G'}(v)=\varnothing]=\left(1-\frac{1}{\sqrt{n\log n}}\right)^{|N_{G'}(v)\cap A|}\leq\nume^{-\log^{3/2}n}.\]
    Again by a union bound, we conclude that a.a.s.\ $S$ dominates $V(G)\setminus B'$ in $G'$, and thus, in particular, it dominates $N_\cC(S)$.

    Since both statements hold a.a.s., it follows that there is a set $S\subseteq A$ which satisfies the properties required for \cref{lem:thomassen}.
    By the choice of $A$, the conclusion follows immediately by applying this lemma.
\end{proof}

\subsection{Crossing and non-crossing edges in graph drawings}\label{sec:drawings}

We will need some auxiliary results about drawings of graphs in the plane.
While there exist much more general results, throughout this paper the word \emph{drawing} refers to a drawing of a graph on the plane where vertices are represented with points and edges are represented with straight-line segments, and such that for every $uv\in E(G)$ there is no other vertex of $G$ on the line going through $u$ and $v$.
Throughout this section, we identify the vertices and edges of a graph with their representations in its drawing.
We say that two edges are \emph{crossing} if they intersect at a point other than one of their endpoints, and that they are \emph{disconnected} if they are not crossing and not incident to each other.

As previously discussed in \cref{sect:C4switch}, we will need to consider whether the edges of an implanted~$C_4$ are in a ``crossing'' or ``non-crossing'' configuration, as this determines whether the corresponding $C_4$-switch can be used to increase the number of components of a $2$-factor.
In general, we will also need to consider whether different collections of implanted $C_4$'s ``cross'' each other or not.
For this reason, we will need sufficient conditions for any drawing of a graph to contain three pairwise crossing or pairwise disconnected edges.
\citet{AAPPS97} showed that any drawing of a sufficiently dense graph contains three pairwise crossing edges.

\begin{lemma}[{\cite[Theorem~1.1]{AAPPS97}}]\label{lm:crossing}
     There exists a constant $c$ such that any drawing of an $n$-vertex graph with at least $cn$ edges contains three pairwise crossing edges.
 \end{lemma}

An analogous statement for disconnected edges was proved by \citet{GKK96}.

\begin{lemma}[{\cite[Theorem~1]{GKK96}}]\label{lm:noncrossing}
    Any drawing of an $n$-vertex graph with at least $3n+1$ edges contains three pairwise disconnected edges.
\end{lemma}

\section{Proof}

We now turn to the proof of \cref{thm:main}. 
Our argument follows the strategy outlined in \cref{sec:sketch} and is divided into three principal stages.
As indicated at the end of \cref{sec:sketch}, the first step is to obtain a partition of the given graph $G$ such that every pair of vertices in the same part have many common neighbours. 
We prove that such a partition exists in \cref{sec:part}.
In the second stage, we prove that $G$ contains a Hamilton cycle with many implanted copies of $C_4$ (see \cref{sec:embed}). 
This forms the core of our argument.
For the third stage, in \cref{sec:switches} we show that, given a Hamilton cycle (or, in more generality, a $2$-factor with at most $k$ cycles) with many implanted copies of $C_4$, one can perform $C_4$-switches to obtain a $2$-factor with the prescribed number of cycles~$k$.
Finally, in \cref{sec:main_proof}, we combine these ingredients to complete the proof of \cref{thm:main}.

\subsection{Partitioning the graph}\label{sec:part}

The next lemma shows that the vertex set of any graph of a sufficiently large minimum degree can be partitioned into few parts such that any set of $m$ vertices in a common part have many common neighbours.
In order to find many implanted $C_4$'s in some Hamilton cycle (see \cref{cor:stage}), we will only need the case $m=2$, which implies that any two vertices in a common part belong to many common~$C_4$'s.
The proof is inspired by the dependent random choice method developed by \citet{AKS03}.

\begin{lemma}\label{lm:part}
    Let $0<1/n\ll\eps\ll \zeta < 1/2$ and $0<1/n\ll \eps \ll 1/m <1$\COMMENT{Taking $\eps\leq \zeta^3/(6m)$ suffices.}.
    Let $G$ be an $n$-vertex graph with $\delta(G)\geq n^{1-\eps}$.
    Then, there exist a positive integer $s\leq n^\zeta$ and a partition $V(G)=V_1\cup \dots \cup V_s$ such that, for all $i\in [s]$ and $S\subseteq V_i$ with $|S|=m$, we have that $\lvert\bigcap_{v\in S}N_G(v)\rvert\geq n^{1-\zeta}+1$.
\end{lemma}

\begin{proof}
    We first claim that there exists a partition $V(G)=A\cup B$ such that $||A|-|B||\leq 1$ and the bipartite subgraph~$G'$ of~$G$ induced by~$A$ and~$B$ has minimum degree
    \begin{equation}\label{equa:G'mindeg}
        \delta(G')\geq n^{1-2\eps}.
    \end{equation}
    Indeed, consider a random equipartition of $V(G)$, obtained by choosing a uniformly random subset $A\subseteq V(G)$ of size $|A|=\lfloor n/2\rfloor$ and letting $B\coloneqq V(G)\setminus A$.
    We then have that $d_{G'}(v)$ follows a hypergeometric distribution with
    \[\mathbb{E}[d_{G'}(v)]=d_G(v)\frac{|A|\pm1}{n}\geq \frac{d_G(v)}{3}\]
    for each $v\in V(G)$.
    By \cref{lem:Chernoff} and a union bound,
    \begin{align*}
        \mathbb{P}[\delta(G')< n^{1-2\eps}]&\leq \sum_{v\in V(G)}\mathbb{P}\left[|d_{G'}(v)-\mathbb{E}[d_{G'}(v)]|>\frac{1}{2}\mathbb{E}[d_{G'}(v)]\right]\\
        &\leq \sum_{v\in V(G)}2\nume^{-\mathbb{E}[d_G(v)]/12}\leq 2n \nume^{-n^{1-\eps}/36}=o(1).
    \end{align*}
    
    By symmetry, it now suffices to find a partition of $B$ into parts of size at least $n^{1-\zeta}$ which satisfy the desired property.
    Partitioning~$A$ analogously will then immediately yield a partition 
    \[V(G)=A\cup B=V_1\cup \dots \cup V_s\]
    with $s\leq n^\zeta$ which satisfies the desired property.
    
    Let $\ell\coloneqq\lceil2m/\zeta\rceil$.
    Given a set $M\subseteq A$ of size $|M|=n^{6\eps}$, we say that a subset $S\subseteq B$ of size $m$ is \emph{bad} with respect to $M$ if $\lvert\bigcap_{v\in S}N_{G'}(v)\rvert\leq n^{1-\zeta}$ but $\lvert\bigcap_{v\in S}N_{G'}(v)\cap M\rvert\geq \ell$.

    \begin{claim}\label{claim:M}
        There exists a set $M\subseteq A$ of size $|M|=n^{6\eps}$ such that $|N_{G'}(v)\cap M|\geq n^{3\eps}$ for all $v\in B$ and such that~$B$ has no subset of size $m$ which is bad with respect to~$M$.
    \end{claim}

    \begin{claimproof}
        Let $M$ be a subset of~$A$ of size $|M|=n^{6\eps}$ chosen uniformly at random.
        For each $v\in B$, let \mbox{$X_v\coloneqq|N_{G'}(v)\cap M|$}.
        We have that $X_v$ follows a hypergeometric distribution with $\mathbb{E}[X_v]\geq n^{4\eps}$, so \cref{lem:Chernoff} implies that
        \[
            \mathbb{P}[X_v< n^{3\eps}]\leq \mathbb{P}\left[X_v<\frac12 \mathbb{E}[X_v]\right]\leq \mathbb{P}\left[|X_v-\mathbb{E}[X_v]|> \frac12\mathbb{E}[X_v]\right]\leq 2\nume^{-n^{4\eps}/12}.
        \]
        It follows immediately by a union bound that a.a.s.\ $|N_{G'}(v)\cap M|=X_v\geq n^{3\eps}$ for all $v\in B$.

        Now we want to bound the probability that $B$ contains a subset which is bad with respect to~$M$.
        Fix a set $S\subseteq B$ of size $|S|=m$, and assume that $\lvert\bigcap_{v\in S}N_{G'}(v)\rvert\leq n^{1-\zeta}$.\COMMENT{Otherwise, we already know that $S$ is not bad for $M$.}
        Then, fix a set $L\subseteq\bigcap_{v\in S}N_{G'}(v)$ of size $|L|=\ell$.
        The probability that $S$ is bad for $M$ because $L\subseteq M$ is\COMMENT{We have
        \[\frac{\binom{n/2}{n^{6\eps}-\ell}}{\binom{n/2}{n^{6\eps}}}=\frac{\frac{(n/2)_{n^{6\eps}-\ell}}{(n^{6\eps}-\ell)!}}{\frac{(n/2)_{n^{6\eps}}}{(n^{6\eps})!}}=\frac{(n^{6\eps})_\ell}{(n/2-n^{6\eps}+\ell)_\ell}\leq\frac{n^{6\eps\ell}}{\frac12\left(\frac{n}{2}\right)^{\ell}}=2\cdot2^\ell n^{(6\eps-1)\ell}\]
        (where for the inequality we assume that $n$ is sufficiently large).}
        \[\frac{\binom{n/2}{n^{6\eps}-\ell}}{\binom{n/2}{n^{6\eps}}}\leq2\cdot2^\ell n^{(6\eps-1)\ell}.\]
        By a union bound over the at most $\inbinom{n/2}{m}$ sets $S\subseteq B$ of size $|S|=m$ with $\lvert\bigcap_{v\in S}N_{G'}(v)\rvert\leq n^{1-\zeta}$ and over the at most $\inbinom{n^{1-\zeta}}{\ell}$ sets of size $\ell$ within $\bigcap_{v\in S}N_{G'}(v)$, we conclude that the probability that~$B$ contains a bad subset with respect to~$M$ is at most\COMMENT{The upper bounds are immediate using the hierarchy, and noting that the additional factor of $2$ above can easily be cancelled with a factor coming from $\inbinom{n/2}{m}$.}
        \[\binom{n/2}{m}\binom{n^{1-\zeta}}{\ell} \frac{\binom{n/2}{n^{6\eps}-\ell}}{\binom{n/2}{n^{6\eps}}}\leq n^m \cdot n^{(1-\zeta)\ell}\cdot 2^\ell n^{(6\eps-1)\ell}\leq n^{-1}.\]
        Therefore, a set $M$ with the desired properties must exist.
    \end{claimproof}

    Let $M\subseteq A$ be a set with the properties guaranteed by \cref{claim:M}.

    \begin{claim}\label{claim:R}
        For any set $T\subseteq M$ of size $|T|=n^{3\eps}$, there exists a set $R\subseteq B$ of size $|R|=n^{1-\zeta^2}$ such that $\lvert\bigcap_{v\in R}N_{G'}(v)\cap T\rvert\geq \ell$.
    \end{claim}

    \begin{claimproof}
        By \eqref{equa:G'mindeg}, we have that
        \[e_{G'}(T, B)\geq n^{1+\eps}.\]
        Now suppose that the statement is false, that is, that for every set $R\subseteq B$ of size $|R|=n^{1-\zeta^2}$ we have that $\lvert\bigcap_{v\in R}N_{G'}(v)\cap T\rvert<\ell$.
        In other words, every set $L\subseteq T$ of size $|L|=\ell$ has fewer than $n^{1-\zeta^2}$ common neighbours in $B$.
        This implies that\COMMENT{By adding over all sets $L$, we are counting every edge which is contained in a copy of $K_{1,\ell}$ (with the vertex on the side of size $1$ contained in $B$), possibly multiple times; to account for all edges which are not contained in such stars, we add the term $|B|(\ell-1)$.
        For the last inequality, we use the hierarchy (roughly, it suffices that $\eps<\zeta^3/6m$).}
        \begin{align*}
            e_{G'}(T, B)&< \sum_{L\in \binom{T}{\ell}}n^{1-\zeta^2}\ell+|B|(\ell-1)
            = \binom{n^{3\eps}}{\ell}n^{1-\zeta^2} \ell+|B|(\ell-1)\leq n^{1+3\eps \ell-\zeta^2}\ell+n \ell
            <n^{1+\eps},
        \end{align*}
        a contradiction.
    \end{claimproof}

    Consider the auxiliary graph $H$ on $B$ where two vertices form an edge if and only if they have at least $\ell$ common neighbours in $M$.
    We now multicolour the edges of $H$ with colour palette~$\inbinom{M}{\ell}$, and each edge $uv\in E(H)$ receives all colours in $\inbinom{N_{G'}(u)\cap N_{G'}(v)\cap M}{\ell}$.
    We say that a vertex $v\in B$ is \emph{good} with respect to a colour $L\in \inbinom{M}{\ell}$ if $v$ belongs to a monochromatic clique of colour $L$ of order at least~$n^{1-\zeta^2}$.
    Since every vertex of~$B$ has at least $n^{3\eps}$ neighbours in~$M$ by \cref{claim:M}, it follows from \cref{claim:R} that every vertex is good with respect to at least one colour.\COMMENT{Indeed, fix $v\in B$ and consider any set $T\subseteq N_{G'}(v)\cap M\subseteq M$ of size $|T|=n^{3\eps}$ (which exists by \cref{claim:M}).
    By \cref{claim:R}, there exists a set $R\subseteq B$ of size $|R|=n^{1-\zeta^2}$ such that $\lvert\bigcap_{u\in R}N_{G'}(u)\cap T\rvert\geq\ell$, that is, all pairs of vertices $u,u'\in R$ share a common colour (in fact, they share a common subset of colours: each of the $\ell$-sets contained in $\bigcap_{u\in R}N_{G'}(u)\cap T$, of which there exists at least one), so they form a monochromatic clique of size $n^{1-\zeta^2}$.
    Now, it could be that $v\notin R$, but since $T\subseteq N_{G'}(v)$, it follows that $uv$ also receives that same colour for all $u\in R$, and so $v$ is contained in a sufficiently large monochromatic clique.}
    For each $v\in B$, fix one such colour $L_v$ uniformly at random and independently from each other, and let $B_1\cup \dots \cup B_r$, for some $r\in\inbinom{|M|}{\ell}$, be the partition of~$B$ where two vertices $u,v\in B$ belong to a common part if and only if $L_u=L_v$.
    Since~$B$ contains no sets of size~$m$ which are bad with respect to~$M$ by \cref{claim:M}, each part~$B_i$ satisfies the desired property that, for every $S\subseteq B_i$ with $|S|=m$, we have that $\lvert\bigcap_{v\in S}N_G(v)\rvert\geq n^{1-\zeta}+1$.\COMMENT{Indeed, the properties of $M$ (there being no bad sets wrt $M$) guarantee that for any set $S\subseteq B$ of size $|S|=m$ with $\lvert\bigcap_{v\in S}N_{G'}(v)\rvert\leq n^{1-\zeta}$ we must have $\lvert\bigcap_{v\in S}N_{G'}(v)\cap M\rvert<\ell$.
    But for any set $S\subseteq B_i$ we have proved that $\lvert\bigcap_{v\in S}N_{G'}(v)\cap M\rvert\geq\ell$ (because, in particular, all their neighbourhoods contain the same $\ell$-set $L_v$), so we must have $\lvert\bigcap_{v\in S}N_{G'}(v)\rvert\geq n^{1-\zeta}+1$, as desired.}
    Moreover, as each colour which can potentially be chosen for a vertex can be chosen for at least $n^{1-\zeta^2}$ vertices, we have that\COMMENT{For the first inequality, simply notice that $|B_i|$ is a sum of independent Bernoulli random variables, each indicating, for each vertex $v\in B$, whether the colour picked for it is the colour $L$ that defines $B_i$ or not.
    Many of these variables may in fact be constants equal to $0$ (if the vertex is not good with respect to $L$), and for each other the probability of success is $1$ divided by the number of colours with respect to which the vertex is good; this is always bounded from below by $1/\inbinom{|M|}{\ell}$, and if there is one vertex which is good for $L$ (only such colours are relevant for the partition), then there are at least $n^{1-\zeta^2}$ such vertices (by the definition of them being good), thus yielding the lower bound on the expected size.
    The second bound follows by using the hierarchy (we require roughly $\eps<\zeta^3/6m$).}
    \[\mathbb{E}[|B_i|]\geq \frac{n^{1-\zeta^2}}{\binom{|M|}{\ell}}\geq n^{1-2\zeta^2}\]
    for each $i\in [r]$ and so, by \cref{lem:Chernoff} and the bound on $\zeta$, we have that
    \[\mathbb{P}[|B_i|< n^{1-\zeta}]\leq\mathbb{P}\left[|B_i|< \frac12\mathbb{E}[|B_i|]\right]\leq 2\nume^{-\mathbb{E}[|B_i|]/12}\leq 2\nume^{-n^{1-2\zeta^2}/12}.\]
    \COMMENT{We need this to tend to $0$ (faster than any polynomial, say), so we need to assume that $1-2\zeta^2>0\iff\zeta<1/\sqrt{2}$.
    We are also assuming that $n^{1-\zeta}=o(n^{1-2\zeta^2})$ for the first inequality, which only holds if $\zeta<1/2$.} 
    By a union bound over all $i\in[r]$, since $r\leq\inbinom{n^{6\eps}}{\ell}$, we conclude that a.a.s.\ $|B_i|\geq n^{1-\zeta}$ for all $i\in [r]$.
    Thus, there exists a partition $B=B_1\cup \dots \cup B_r$ where each part satisfies the desired property and has size at least $n^{1-\zeta}$, as desired.
\end{proof}

\subsection{Finding a Hamilton cycle with many implanted \texorpdfstring{$C_4$}{C4}'s}\label{sec:embed}

In this section, we prove the main ingredient of our proof, which states that we can find a Hamilton cycle with many implanted $C_4$'s.
Given a graph $G$ and a $2$-factor $\cC$ of $G$, we denote by $H_{\cC}$ the auxiliary graph on $E(\cC)$ where distinct non-incident edges $e,e'\in E(\cC)$ are adjacent in $H_{\cC}$ if and only if $G$ has a $C_4$ which contains both~$e$ and~$e'$ and no other edge of $\cC$.\COMMENT{When $\cC$ contains three or four edges of a $C_4$, we cannot do a switch.}
Thus, $e(H_{\cC})$ counts the number of~$C_4$'s which are implanted in~$\cC$.

\begin{lemma}\label{cor:stage}
    Let $0<1/n\ll\eps\ll \eta \leq 1$\COMMENT{Relationship between $\eps$ and $\eta$ mostly inherited from \cref{lm:part}, $\eps<\eta^3/720$ would work.}.
    Let $G$ be a Hamiltonian $n$-vertex graph with $\delta(G)\geq n^{1-\eps}$.
    Let $\cC$ be a Hamilton cycle of~$G$ and $E_0\subseteq E(\cC)$ be a set of size $|E_0|\leq n^{1-\eta}$.
    Then $G$ has a Hamilton cycle~$\cC'$ such that $E_0\subseteq E(\cC')$ and $e(H_{\cC'})\geq n^{2-\eta}$.
\end{lemma}

The remainder of this section is devoted to proving \Cref{cor:stage}.
Recall from \cref{sec:sketch} that our strategy is roughly as follows.
We let $\cC$ be a Hamilton cycle which contains a maximum number of ``good sets'' and, subject to that, a maximum number of ``good edges''.
The case where $\cC$ contains sufficiently many good edges will correspond to the case where $\cC$ has many implanted $C_4$'s.
Therefore, we will suppose for a contradiction that it is not the case and we will want to use \cref{lem:thomassen_new} to find a new cycle $\cC'$ which contradicts the maximality of~$\cC$.
For its application, we need a dense graph $G'$ of edges which would ``improve'' $\cC$ if incorporated.
The next two lemmas construct such a $G'$\,---\,the case where $\cC$ fails to be good enough because it has too few good sets corresponds to \cref{lm:covervtxs}, and the case where $\cC$ has many good sets but too few good edges corresponds to \cref{lm:closeC4}.

We first need some notation, which corresponds to the notion of ``good set'' used in the proof overview.
Given $\zeta>0$, an $n$-vertex graph $G$, and an edge $xy\in E(G)$, we denote by $M_{G,\zeta}(xy)$ the set of vertices $z\in V(G)\setminus \{x,y\}$ which have at least $n^{1-\zeta}$ neighbours $w\in V(G)\setminus \{x,y\}$ such that $xy$ and $zw$ form a $C_4$ in $G$ (that is, both $yz,wx\in E(G)$ or both $yw,zx\in E(G)$).
Informally, $M_{G,\zeta}(xy)$ consists of vertices which form ``many'' $C_4$'s with the edge $xy$.
Given a set $E\subseteq E(G)$, we denote $M_{G,\zeta}(E)\coloneqq \bigcup_{e\in E}M_{G,\zeta}(e)$ (in particular, $M_{G,\zeta}(\varnothing)=\varnothing$).
The ``good sets'' $E$ in \cref{sec:sketch} correspond to sets with very small $V(G)\setminus M_{G,\zeta}(E)$.

\begin{lemma}\label{lm:covervtxs}
    Let $0<1/n\ll \zeta\leq \eta/2 < 1/4$.
    Let $G$ be an $n$-vertex graph, and suppose that there exist $T\subseteq S\subseteq V(G)$ of sizes $|S|\geq |T|\geq n^{1-\eta}$ such that for any distinct $u,v\in S$ we have that $|N_G(u)\cap N_G(v)|\geq n^{1-\zeta}+1$.
    Then, there exists some $G'\subseteq G$ which satisfies the following properties:
    \begin{enumerate}[label=\rm(\roman*)]
        \item\label{lm:covervtxs-deg} For each $v\in S$, we have that $d_{G'}(v)\geq n^{1-2\zeta}$.
        \item\label{lm:covervtxs-manyC4} For each $e\in E(G')$, we have that $|M_{G,\zeta}(e)\cap T|\geq n^{1-2\eta}$.
    \end{enumerate}
\end{lemma}

\begin{proof}
    For each $v\in S$, let $N_v\coloneqq\{u\in N_G(v):|N_G(u)\cap T|\geq |T|^{1-2\zeta}\}$.
    We claim that the graph~$G'$ on~$V(G)$ defined by $E(G')\coloneqq \bigcup_{v\in S}\{vu: u\in N_v\}$ satisfies the desired properties.
    
    Indeed, for $v\in S$, we have that
    \[\sum_{w\in T}|N_G(w)\cap N_G(v)|=\sum_{u\in N_G(v)}|N_G(u)\cap T|\]
    (this holds since both expressions count the number of walks of length~$2$ in~$G$ with one endpoint in~$v$ and the other in~$T$).\COMMENT{Indeed, we consider \emph{walks} rather than \emph{paths} for we must account for the case that $v\in T$.
    With this double-counting idea in mind, observe that the LHS counts such walks from the endpoint in $T$, whereas the RHS counts the walks from the middle point, which must be a neighbour of $v$.}
    It then follows that\COMMENT{The first inequality is immediate by the assumption on the common neighbourhoods of vertices in $S$.
    The next inequality uses the fact that vertices $u\in N_v$ have at most $|T|$ neighbours within $T$, so this terms cancels with the denominator, whereas vertices not in $N_v$ have at most $|T|^{1-2\zeta}$ such neighbours by definition.
    The last inequality simply uses the definition of $G'$ (and is an inequality because, a priori, it could be that some vertex $v\in S$ belongs to $N_w$ for some other $w\in S$ for which $w\notin N_v$).}
    \[n^{1-\zeta}\leq \frac{\sum_{w\in T}|N_G(w)\cap N_G(v)|}{|T|}= \frac{\sum_{u\in N_G(v)}|N_G(u)\cap T|}{|T|}\leq |N_v|+n|T|^{-2\zeta}\leq d_{G'}(v) +n|T|^{-2\zeta},\]
    and so $d_{G'}(v)\geq n^{1-\zeta}-n|T|^{-2\zeta}\geq n^{1-\zeta}-n^{1-2\zeta(1-\eta)}\geq n^{1-2\zeta}$\COMMENT{Indeed, note that $n|T|^{-2\zeta}\leq n^{1-2\zeta(1-\eta)}=n^{1-2\zeta+2\zeta\eta}$ by the lower bound on $|T|$, where this exponent satisfies that $1-2\zeta+2\zeta\eta<1-\zeta$ by the bound on $\eta$, and thus for sufficiently large $n$ we have that $n^{1-\zeta}-n|T|^{-2\zeta}\geq n^{1-\zeta}/2\geq n^{1-2\zeta}$.} and \ref{lm:covervtxs-deg} holds.
    
    Moreover, fix any $uv\in E(G')$, and suppose that $v\in S$ and $u\in N_v$. 
    Note that $M_{G,\zeta}(uv)$ contains all neighbours $z$ of $u$ (other than $v$) such that $|N_G(v)\cap N_G(z)|\geq n^{1-\zeta}+1$.\COMMENT{Here, with the $+1$ we account for $u$ itself, which does not count towards constructing the desired copies of $C_4$.}
    Since $T\subseteq S$, any $z\in N_G(u)\cap T$ has at least $n^{1-\zeta}+1$ common neighbours with $v$ by assumption, and thus they all lie in $M_{G,\zeta}(uv)$.
    Therefore,\COMMENT{For the first inequality, note that the $-1$ comes from the possibility that $v\in T$.
    The second inequality holds by the definition of $N_v$ (recall that $u\in N_v$); the rest are simple calculations (but note that, for the third, we need that $1-2\zeta\geq0$, as otherwise the inequality would revert, hence an upper bound on $\zeta$, and for the last we need that $(1-\eta)(1-2\zeta)=1-\eta-2\zeta+2\eta\zeta>1-2\eta\iff(1+2\zeta)\eta>2\zeta\iff\eta>\frac{2\zeta}{1+2\zeta}$).}
    \[|M_{G,\zeta}(uv)\cap T| \geq |N_G(u)\cap T|-1 \geq |T|^{1-2\zeta}-1\geq n^{(1-\eta)(1-2\zeta)}-1\geq n^{1-2\eta}\]
    and \ref{lm:covervtxs-manyC4} holds, as desired.
\end{proof}

\begin{lemma}\label{lm:closeC4}
    Let $0< 1/n \ll \zeta < \eta < 1/4$.
    Let $G$ be an $n$-vertex graph and $S\subseteq V(G)$.
    Suppose that there exist disjoint $E_1, \dots, E_{n^{1-3\eta}}\subseteq E(G)$ such that $|S\setminus M_{G,\zeta}(E_i)|\leq n^{1-\eta}$ for each $i\in [n^{1-3\eta}]$.
    Then, there exist $G'\subseteq G$ and $B\subseteq S$ with the following properties:
    \begin{enumerate}[label=\rm(\roman*)]
        \item $|B|\leq 2n^{1-\eta}$.\label{lm:closeC4-B}
        \item For each $v\in S\setminus B$, we have that $d_{G'}(v)\geq n^{1-2\zeta}$.\label{lm:closeC4-deg}
        \item For each $e\in E(G')$, there are at least $n^{1-4\eta}$ edges $e'\in\bigcup_{i\in [\ell]}E_i$ such that $e$ and $e'$ form a $C_4$ in~$G$.\label{lm:closeC4-manyC4}
    \end{enumerate}
\end{lemma}

\begin{proof}
    Let $B$ be the set of vertices $v\in S$ for which there exist at least ${n^{1-3\eta}}/{2}$ indices $i\in [n^{1-3\eta}]$ such that $v\notin M_{G,\zeta}(E_i)$.
    If we assume that $|B|>2n^{1-\eta}$, then
    \[n^{2-4\eta}<|B|\frac{n^{1-3\eta}}{2}\leq\left\lvert\{(v,i)\in S\times[n^{1-3\eta}]:v\notin M_{G,\zeta}(E_i)\}\right\rvert=\sum_{i\in[n^{1-3\eta}]}|S\setminus M_{G,\zeta}(E_i)|\leq n^{2-4\eta},\]
    a contradiction.
    So $|B|\leq 2n^{1-\eta}$,
    as desired for \cref{lm:closeC4-B}.

    For each $v\in S\setminus B$ and $i\in [n^{1-3\eta}]$, let $N_{v,i}$ be the set of vertices $u\in N_G(v)$ such that there exists some $e\in E_i$ which forms a $C_4$ with $uv$ in $G$.
    For each $v\in S\setminus B$, let $N_v$ denote the set of vertices $u\in N_G(v)$ for which there exist at least $n^{1-4\eta}$ indices $i\in [n^{1-3\eta}]$ such that $u\in N_{v,i}$.
    We claim that the graph $G'$ on $V(G)$ defined by $E(G')\coloneqq \bigcup_{v\in S\setminus B}\{uv: u\in N_v\}$ satisfies the desired properties.
    
    Indeed, \cref{lm:closeC4-manyC4} holds by construction.\COMMENT{This uses crucially the assumption that the edge-sets $E_i$ are disjoint, as otherwise it could be that the edge coming from different indices is always the same.}
    For \cref{lm:closeC4-deg}, observe first that, by the definition of $M_{G,\zeta}(\cdot)$, for $v\in S\setminus B$ and $i\in [n^{1-3\eta}]$, if $v\in M_{G,\zeta}(E_i)$, then $|N_{v,i}|\geq n^{1-\zeta}$.
    Therefore, for all $v\in S\setminus B$ we have that
    \[\sum_{i\in [n^{1-3\eta}]}|N_{v,i}|\geq\frac{n^{1-3\eta}}{2}\cdot n^{1-\zeta}.\]
    On the other hand, by the definition of $N_v$, when considering the same sum as above, each $u\notin N_v$ is counted at most $n^{1-4\eta}$ times.
    Therefore,
    \[\sum_{i\in [n^{1-3\eta}]}|N_{v,i}|\leq |N_v|n^{1-3\eta}+\left|\bigcup_{i\in [n^{1-3\eta}]}N_{v,i}\right| n^{1-4\eta} \leq d_{G'}(v)n^{1-3\eta}+n^{2-4\eta}.\]
    These two inequalities imply that $d_{G'}(v)\geq {n^{1-\zeta}}/{2}-n^{1-\eta}\geq n^{1-2\zeta}$, as desired.
\end{proof}

We are now ready to prove \cref{cor:stage}.

\begin{proof}[Proof of \cref{cor:stage}]
    Let $0<1/n\ll\eps\ll\zeta\ll \eta \leq 1$\COMMENT{One can track dependencies here. Crude bound: $\zeta< \eta/60$ would suffice.} and $\eta'\coloneqq \eta/10$.
    Apply \cref{lm:part} with $m=2$ to obtain a partition $V(G)=V_1\cup \dots \cup V_s$ of $V(G)$ into $s\leq n^\zeta$ parts such, for all $i\in [s]$, any two vertices $u,v\in V_i$ satisfy that $|N_G(u)\cap N_G(v)|\geq n^{1-\zeta}+1$.

    Given a Hamilton cycle $\cC_0\subseteq G$ and $i\in [s]$, let $t_i(\cC_0)$ be the largest integer in $[0,n^{1-3\eta'}]$ such that there exist (possibly empty) edge sets $E_{i,1}, \dots, E_{i,t_i(\cC_0)+1}$ such that,  denoting $E_i\coloneqq E_{i,1}\cup \dots\cup E_{i,t_i(\cC_0)+1}$, the following properties hold:
    \begin{enumerate}[label=\rm(\roman*)]
        \item $E_i\subseteq E(\cC_0)$.\label{cor:stage-C}
        \item For all distinct $j,k\in[t_i(\cC_0)+1]$ we have that $E_{i,j}\cap E_{i,k}=\varnothing$.\label{cor:stage-disjoint}
        \item For all $j\in [t_i(\cC_0)+1]$, if $j\in[n^{1-3\eta'}]$, then $|E_{i,j}|\leq n^{2\eta'}$, and if $j= n^{1-3\eta'}+1$, then $|E_{i,j}|\leq n^{1-6\eta'}$.\label{cor:stage-size}
        \item For all $j\in [t_i(\cC_0)]$ we have that $|V_i\setminus M_{G,\zeta}(E_{i,j})|\leq n^{1-\eta'}$.\label{cor:stage-Mfull}
        \item If $t_i(\cC_0)+1\leq n^{1-3\eta'}$, then 
        \[|V_i\setminus M_{G,\zeta}(E_{i,t_i(\cC_0)+1})|>n^{1-\eta'}\quad\text{and}\quad|M_{G,\zeta}(E_{i,t_i(\cC_0)+1})\cap V_i|\geq |E_{i,t_i(\cC_0)+1}|\cdot n^{1-2\eta'}.\]
        In particular, $|E_{i,t_i(\cC_0)+1}|<n^{2\eta'}$.\COMMENT{Assume not. From the second inequality above, it must be the case that $|M_{G,\zeta}(E_{i,t_i(\cC_0)+1})\cap V_i|\geq |E_{i,t_i(\cC_0)+1}|\cdot n^{1-2\eta'}=n$, but this contradicts the first inequality above.}\label{cor:stage-Mpartial}
        \item If $t_i(\cC_0)= n^{1-3\eta'}$, then $e(H_{\cC_0}[E_i])\geq |E_{i,t_i(\cC_0)+1}|\cdot n^{1-4\eta'}$.
        In particular, if $e(H_{\cC_0})< n^{2-10\eta'}$, then this implies that $|E_{i,t_i(\cC_0)+1}|<n^{1-6\eta'}$.\label{cor:stage-C4}
    \end{enumerate}
    (Note that $t_i(\cC_0)$ is well-defined since, if $|V_i|>n^{1-\eta'}$, then the above properties hold with $0$ and $\varnothing$ playing the roles of $t_i(\cC_0)$ and $E_{i,t_i(\cC_0)+1}$, and if $|V_i|\leq n^{1-\eta'}$, then $t_i(\cC_0)=n^{1-3\eta'}$ and the properties hold with $\varnothing$ playing the role of $E_{i,j}$ for each $j\in [t_i(\cC_0)+1]$.) 
    Moreover, let $m_i(\cC_0)$ be the largest integer such that there exist sets $E_{i,1}, \dots, E_{i,t_i(\cC_0)+1}$ satisfying the above properties and such that $|E_{i,t_i(\cC_0)+1}|=m_i(\cC_0)$.

    Let $\cC'\subseteq G$ be a Hamilton cycle with $E_0\subseteq E(\cC')$ which maximises $\sum_{i\in [s]}t_i(\cC')$ and, subject to this, maximises $\sum_{i\in [s]}m_i(\cC')$.
    Suppose for a contradiction that $e(H_{\cC'})< n^{2-10\eta'}$.
    We will use \cref{lem:thomassen_new} to prove the existence of a new Hamilton cycle $\cC''$ which contradicts the maximality $\cC'$.

    For each $i\in [s]$, denote $t_i\coloneqq t_i(\cC')$ and $m_i\coloneqq m_i(\cC')$, and fix sets $E_{i,1}, \dots, E_{i,t_i+1}$ such that properties \cref{cor:stage-C,cor:stage-disjoint,cor:stage-Mfull,cor:stage-C4,cor:stage-Mpartial,cor:stage-size} hold and $|E_{i,t_i+1}|=m_i$.
    Denote $E\coloneqq\bigcup_{i=0}^sE_i$.
    By relabelling if necessary, we may assume that $t_1\leq \ldots \leq t_s$.
    Let $r\in[0,s]$ be the largest integer such that $t_i+1\leq n^{1-3\eta'}$ for all $i\in [r]$.

    Note that, by \cref{cor:stage-size} and the assumption from the statement, we have that\COMMENT{We have
    \[|V(E)|\leq 2n^{1-10\eta'}+2sn^{1-3\eta'}n^{2\eta'}+2sn^{1-6\eta'}\leq2(n^{1-10\eta'}+n^{1-\eta'+\zeta}+n^{1-6\eta'+\zeta})\leq n^{1-\eta'/2},\]
    where for the last inequality it suffices that $\zeta<\eta'/2$ (for sufficiently large $n$).}
    \begin{equation}\label{equa:Bbound1}
        |V(E)|\leq 2n^{1-10\eta'}+2sn^{1-3\eta'}n^{2\eta'}+2sn^{1-6\eta'}\leq n^{1-\eta'/2}.
    \end{equation}
    For each $i\in [r]$, recall from \cref{cor:stage-Mpartial} that $|V_i\setminus M_{G,\zeta}(E_{i,t_i+1})|>n^{1-\eta'}$, so we may let $G_i\subseteq G$ be the graph obtained by applying \cref{lm:covervtxs} with $\eta'$, $V_i$, and $V_i\setminus M_{G,\zeta}(E_{i,t_i+1})$ playing the roles of $\eta$, $S$, and~$T$.
    For each $i\in [s]\setminus [r]$, by \cref{cor:stage-Mfull}, we may let $G_i\subseteq G$ and $B_i\subseteq V_i$ be the graph and vertex set obtained by applying \cref{lm:closeC4} with $\eta'$, $V_i$, and $E_{i,j}$ playing the roles of $\eta$, $S$, and~$E_j$, respectively.
    Let $G'\coloneqq \bigcup_{i\in [s]} G_i$ and \mbox{$B'\coloneqq V(E)\cup \bigcup_{i\in [s]\setminus [r]}B_i$}.
    Now, by~\eqref{equa:Bbound1} and \cref{lm:closeC4}~\ref{lm:closeC4-B}, we have that\COMMENT{Here we need again that $\zeta<2\eta'/3$, say.}
    \[|B'|\leq n^{1-\eta'/2}+2sn^{1-\eta'}\leq n^{1-\eta'/3}.\]
    Moreover, by \cref{lm:covervtxs}~\ref{lm:covervtxs-deg} and \cref{lm:closeC4}~\ref{lm:closeC4-deg}, for every vertex \mbox{$v\in V(G)\setminus B'$} we have that\COMMENT{For the last inequality, for sufficiently large $n$, it suffices to verify that $1-2\zeta>1/2$, for which it suffices that $\zeta<1/4$ (which holds), and that $1-2\zeta>1-\eta'/3$, for which we need that $\zeta<\eta'/6$.}
    \[d_{G'}(v)\geq n^{1-2\zeta}\geq \sqrt{n}\log^2n+3|B'|+2.\]
    We are thus in a position where we may apply \cref{lem:thomassen_new} using $E$ as the set of ``protected'' edges and~$E(G')$ as the set of ``good'' edges which we may incorporate into a new Hamilton cycle.
    Let $\cC''$ be the Hamilton cycle obtained by applying \cref{lem:thomassen_new} with $G'$, $\cC'$, and~$\bigcup_{i\in [s]\setminus [r]}B_i$ playing the roles of~$G$, $\cC$, and~$B$, respectively.
    We claim that $\cC''$ contradicts the maximality of $\cC'$.
    The rest of the proof is devoted to showing this.
    
    By \cref{lem:thomassen_new}, $E\subseteq V(\cC'')$ and so, for each $i\in[s]$, the sets $E_{i,1}, \dots, E_{i, t_i+1}$ satisfy properties \cref{cor:stage-C,cor:stage-disjoint,cor:stage-Mfull,cor:stage-C4,cor:stage-Mpartial,cor:stage-size} with $\cC''$ playing the role of~$\cC_0$ in \cref{cor:stage-C}, $H_{\cC''}$ playing the role of~$H_{\cC_0}$ in~\ref{cor:stage-C4}, and $t_i$ instead of~$t_i(\cC'')$ everywhere.
    That is, they witness that $t_i(\cC'')\geq t_i$ and that, if $t_i(\cC'')= t_i$, then $m_i(\cC'')\geq m_i$.
    Thus, to reach the desired contradiction it suffices to find $i\in [s]$ such that $t_i(\cC'')> t_i$ or $m_i(\cC'')> m_i$.

    Let $e\in E(\cC'')\setminus E(\cC')$ and $i\in [s]$ such that $e\in E(G_i)$.
    Define $E_{i,t_i+1}'\coloneqq E_{i,t_i+1}\cup \{e\}$.
    By \cref{lem:thomassen_new} and property \cref{cor:stage-C}, we have that $E_{i,1}\cup \dots \cup E_{i,t_i}\cup E_{i,t_i+1}'\subseteq E(\cC'')$.
    Moreover, properties~\cref{cor:stage-disjoint,cor:stage-C} imply that $E_{i,t_i+1}'\cap E_{i,j}=\varnothing$ for all $j\in [t_i]$.
    By assumption and the ``in particular'' parts of \cref{cor:stage-Mpartial,cor:stage-C4}, we have that $|E_{i,t_i+1}'|\leq n^{2\eta'}$ if $t_i+1\leq n^{1-3\eta'}$ and $|E_{i,t_i+1}'|\leq n^{1-6\eta'}$ if $t_i= n^{1-3\eta'}$.
    If $t_i= n^{1-3\eta'}$ (and so $t_i(\cC'')=t_i$), recalling that $E_i\subseteq E(\cC')\cap E(\cC'')$ by \cref{lem:thomassen_new}, we have by \cref{lm:closeC4}~\ref{lm:closeC4-manyC4} and property~\cref{cor:stage-C4} that
    \[e(H_{\cC''}[E_i\cup\{e\}])\geq e(H_{\cC'}[E_i])+n^{1-4\eta'}\geq |E_{i,t_i+1}'|n^{1-4\eta'},\]
    and so $E_{i,1}, \dots, E_{i, t_i},E_{i, t_i+1}'$ witness that $m_i(\cC'')\geq m_i+1$, a contradiction.
    We may therefore assume that $t_i+1\leq n^{1-3\eta'}$.
    If $|V_i\setminus M_{G,\zeta}(E_{i,t_i+1}')|\leq n^{1-\eta'}$, then $E_{i,1}, \dots, E_{i, t_i},E_{i, t_i+1}',\varnothing$ witness that $t_i(\cC'')\geq t_i+1$, a new contradiction.
    Hence, we may assume that $|V_i\setminus M_{G,\zeta}(E_{i,t_i+1}')|> n^{1-\eta'}$.
    But in this case, by \cref{lm:covervtxs}~\ref{lm:covervtxs-manyC4} (which we applied with $T=V_i\setminus M_{G,\zeta}(E_{i,t_i+1})$) and \cref{cor:stage-Mpartial}, we have that
    \[|M_{G,\zeta}(E_{i,t_i+1}')\cap V_i|\geq |M_{G,\zeta}(E_{i,t_i+1})\cap V_i|+n^{1-2\eta'}\geq |E_{i,t_i+1}'|\cdot n^{1-2\eta'}.\]
    Thus, $E_{i,1}, \dots, E_{i, t_i},E_{i, t_i+1}'$ witness that, if $t_i(\cC'')=t_i$, then $m_i(\cC'')\geq m_i+1$, a final contradiction.\COMMENT{Here we have that $E_{i, t_i+1}'$, of size $m_i+1$, satisfies \cref{cor:stage-Mpartial}.
    However, note that this does not necessarily imply that $m_i(\cC'')\geq m_i+1$.}
    This concludes the proof.
\end{proof}

\subsection{\texorpdfstring{$C_4$}{C4}-switches}\label{sec:switches}

We show here that any Hamilton cycle $\cC$ with many implanted $C_4$'s can be transformed into a $2$-factor which consists of a given number of cycles.

\begin{lemma}\label{lm:switch}
    Let $0< 1/n \ll \eta \ll \delta \leq 1$\COMMENT{$\eta\leq \delta^2/16$}.
    Let $G$ be an $n$-vertex graph which contains a $2$-factor~$\cC$ consisting of exactly $\ell\in[n^{1-\delta}]$ cycles.
    If $e(H_{\cC})\geq n^{2-\eta}$ and $k \in [\ell, n^{1-\delta}]$ is an integer, then~$G$ contains a $2$-factor consisting of exactly\/ $k$ cycles.
\end{lemma} 

\begin{proof}
    We proceed by induction on $k$.
    The result holds trivially for $k=\ell$.
    Let $k>\ell$ and suppose that~$\cC_{k-1}$ is a $2$-factor of $G$ consisting of exactly $k-1$ cycles such that $|E(\cC_{k-1}) \triangle E(\cC)|\leq 12(k-1)$.

    \begin{claim}\label{claim:2cycles}
        There exist (not necessarily distinct) cycles $C$ and $C'$ in $\cC_{k-1}$ and sets $X\subseteq E(C)$ and $Y\subseteq E(C')$, each of size at least $n^{\delta/2}$, such that $\sum_{e\in X}|N_{H_{\cC}}(e)\cap Y|\geq 4|X\cup Y|^{1+\eta}$.
    \end{claim}

    \begin{claimproof}
        Since $e(H_{\cC})\geq n^{2-\eta}$, there exists a set $S\subseteq V(H_{\cC})$ of at least $n^{1-\sqrt{\eta}}$ vertices of $H_{\cC}$ which have degree at least $n^{1-\sqrt{\eta}}$ in $H_{\cC}$.\COMMENT{Otherwise $e(H_{\cC})\leq n\cdot n^{1-\sqrt{\eta}}+n^{1-\sqrt{\eta}}\cdot n= 2n^{2-\sqrt{\eta}}<n^{2-\eta}$.}
        By the pigeonhole principle and the inductive hypothesis, there is a cycle~$C$ in~$\cC_{k-1}$ with $|S\cap E(C)|\geq n^{\delta/2}$.\COMMENT{Given that $S\subseteq E(\cC)$ and $|S|\geq n^{1-\sqrt{\eta}}$, by the inductive hypothesis we have that 
        \[|S\cap E(\cC_{k-1})|\geq n^{1-\sqrt{\eta}}-12(k-1)\geq n^{1-\sqrt{\eta}}-12n^{1-\delta}\geq n^{1-\sqrt{\eta}}/2\]
        (where the last inequality holds by the hierarchy, say).
        Then, by the pigeonhole principle, there is a cycle $C$ in $\cC_{k-1}$ such that 
        \[|S\cap E(C)|\geq\frac{|S\cap E(\cC_{k-1})|}{k-1}\geq\frac{n^{1-\sqrt{\eta}}}{2n^{1-\delta}}=\frac12n^{\delta-\sqrt{\eta}}\leq n^{\delta/2}\]
        (again by the hierarchy; we need $\eta<\delta^2/4$).}
        Fix a set $X\subseteq S\cap E(C)$ of size $|X|=n^{\delta/2}$ arbitrarily.
        It is clear, again from the induction hypothesis, that
        \begin{equation}\label{equa:contrad_switch}
            \sum_{e\in X}d_{H_{\cC_{k-1}}}(e)\geq|X|(n^{1-\sqrt{\eta}}-12(k-1)) \geq n^{1+\delta/4}.
        \end{equation}
        Now, suppose for a contradiction that, for all cycles $C'$ in $\cC_{k-1}$, we have that $|E(C')|<n^{\delta/2}$ or $\sum_{e\in X}|N_{H_{\cC}}(e)\cap E(C')|<|E(C')|^{1+2\eta}$.
        Then,
        \[
            \sum_{e\in X}d_{H_{\cC_{k-1}}}(e)< n^{\delta/2}\cdot(k-1)+\sum_{C'\in \cC_{k-1}}|E(C')|^{1+2\eta}\leq n^{1-\delta/2}+\left(\sum_{C'\in \cC_{k-1}}|E(C')|\right)^{1+2\eta}\leq n^{1+3\eta}.
        \]
        (To see the second inequality, note that the function $f(x)=x^{1+2\eta}$ is convex on the positive reals, so for any positive $a,b\in \mathbb{R}$ we have that
        \begin{align*}
            a^{1+2\eta}+b^{1+2\eta}&= \left(\frac{a}{a+b} (a+b)+\left(1-\frac{a}{a+b}\right)\cdot0\right)^{1+2\eta}+\left(\frac{b}{a+b} (a+b)+\left(1-\frac{b}{a+b}\right)\cdot0\right)^{1+2\eta}\\
            &\leq \frac{a}{a+b}(a+b)^{1+2\eta}+\left(1-\frac{a}{a+b}\right)0^{1+2\eta}+\frac{b}{a+b}(a+b)^{1+2\eta}+\left(1-\frac{b}{a+b}\right)0^{1+2\eta}\\
            &=(a+b)^{1+2\eta},
        \end{align*}
        where the second line holds by definition of convexity.
        This can be applied iteratively to all the summands above to obtain the desired inequality.)
        However, this contradicts \eqref{equa:contrad_switch}.\COMMENT{Here we need $\eta<\delta/12$ (for the global contradiction) and also that $\eta<\delta^2/16$ for the first inequality (for sufficiently large $n$, assuming that the term $k-1$ is much smaller than $n^{1-\sqrt{\eta}}$, which holds under previous assumptions).}
        Therefore, there must exist a cycle $C'$ in $\cC_{k-1}$ with $|E(C')|\geq n^{\delta/2}$ and such that\COMMENT{Formally, we have that $\sum_{x\in X}|N_{H_{\cC}}(x)\cap E(C')|\geq|V(C')|^{1+2\eta}\geq(|E(C')\cup X|/2)^{1+2\eta}\geq 4|X\cup E(C')|^{1+\eta}$, where the last inequality holds, for sufficiently large $n$, by the hierarchy.} 
        \[\sum_{e\in X}|N_{H_{\cC}}(e)\cap E(C')|\geq|E(C')|^{1+2\eta}\geq 4|X\cup E(C')|^{1+\eta}.\]
        To conclude the proof, set $Y\coloneqq E(C')$.
    \end{claimproof}

    Let $C$ and $C'$ be the two cycles of $\cC_{k-1}$ and $X$ and $Y$ be the sets given by \cref{claim:2cycles}.
    Label the vertices of~$C$ and~$C'$ cyclically, say $C=x_1\dots x_{|V(C)|}$ and $C'=y_1\dots y_{|V(C')|}$, in such a way that if $C=C'$ then $x_i=y_i$ for all $i\in [|V(C)|]$.
    For each $i\in [|V(C)|]$, we denote $e_{C,i}\coloneqq x_ix_{i+1}$ (with indices taken modulo~$|V(C)|$).
    Similarly, for each $i\in [|V(C')|]$, let $e_{C',i}\coloneqq y_iy_{i+1}$.

    Let $H'\subseteq H_{\cC}$ be the red/blue-edge-coloured graph on $X\cup Y$ defined as follows.
    The edge set of~$H'$ consists of all the edges $e_{C,i}e_{C',j}\in E(H_{\cC})$ where $e_{C,i}\in X$ and $e_{C',j}\in Y$, and we colour $e_{C,i}e_{C',j}$ red if both $x_iy_j,x_{i+1}y_{j+1}\in E(G)$, and blue otherwise (and so $x_iy_{j+1},x_{i+1}y_j\in E(G)$ in this case).
    Note that, if~$C\neq C'$, then $H'$ is bipartite with bipartition $(X,Y)$.
    Moreover, \cref{claim:2cycles} implies that $e(H')\geq 2|V(H')|^{1+\eta}$.\COMMENT{Note that the loss of a factor of $2$ is for the case that $C=C'$, in which case we would be double counting the edges when summing like in \cref{claim:2cycles}.}
    We now consider four cases to conclude our analysis.
    The more complex cases (which occur when $C\neq C'$) are illustrated in \cref{fig:three}.

    If $C=C'$ and $H'$ has a blue edge $e_{C,i}e_{C,j}$ for some $1\leq i<j\leq |V(C)|$, then
    \[x_ix_{j+1}x_{j+2}\dots x_i\qquad\text{ and }\qquad x_{i+1}x_{i+2}\dots x_j x_{i+1}\]
    form a $2$-factor $\cC'$ of $G[V(C)]$ with $|E(\cC')\triangle E(C)|= 4$, so we are done.
    
    If $C=C'$ and $H'$ has no blue edge, then \cref{lm:crossing}\COMMENT{Draw the vertices on a circle, in the order in which they appear in $C$.} implies that $H'$ has red edges~$e_{C,h}e_{C,j}$ and~$e_{C,i}e_{C,m}$ for some $1\leq h<i<j<m \leq |V(C)|$, and so 
    \[x_hx_jx_{j-1} \dots x_{i+1}x_{m+1}x_{m+2}\dots x_h\qquad\text{ and }\qquad x_{h+1}x_{j+1}x_{j+2} \dots x_m x_ix_{i-1}\dots x_{h+1}\]
    form a $2$-factor $\cC'$ of $G[V(C)]$ with $|E(\cC')\triangle E(C)|= 8$, so we are done.
    
    If $C\neq C'$ and the red subgraph of $H'$ has at least $|X\cup Y|^{1+\eta}$ edges (see also \cref{fig:red}), then \cref{lm:noncrossing}\COMMENT{Draw the vertices of $X$ and $Y$ on two parallel lines, with straight edges.
    Embed the vertices of $X$ in the order which follows $C$ and the vertices of $Y$ in the order which follows $C'$.} implies that there exist $1\leq i_1<i_2<i_3\leq |V(C)|$ and $1\leq j_1<j_2<j_3\leq |V(C')|$ such that $e_{C,i_1}e_{C',j_1}, e_{C,i_2}e_{C',j_2}$, and $e_{C,i_3}e_{C',j_3}$ are red edges of $H'$.
    Then,
    \[x_{i_1}y_{j_1}y_{j_1-1}\dots y_{j_3+1}x_{i_3+1}x_{i_3+2}\dots x_{i_1}\]
    and
    \[x_{i_1+1}y_{j_1+1}y_{j_1+2}\dots y_{j_2}x_{i_2}x_{i_2-1}\dots x_{i_1+1}\]
    and
    \[x_{i_2+1}y_{j_2+1}y_{j_2+2}\dots y_{j_3}x_{i_3}x_{i_3-1}\dots x_{i_2+1}\]
    form a $2$-factor $\cC'$ of $G[V(C)\cup V(C')]$ with $|E(\cC')\triangle (E(C)\cup E(C'))|= 12$, so we are done.

    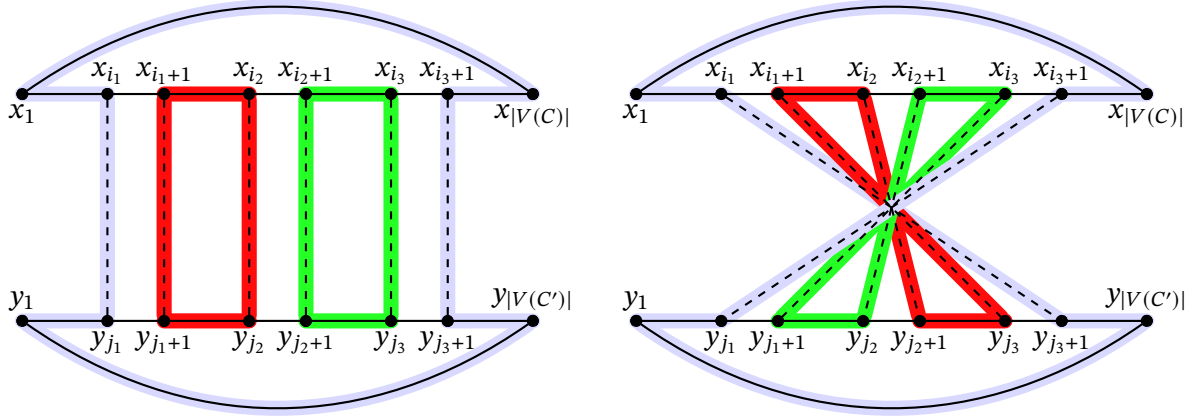
\begin{figure}[h]
        \centering
        \begin{subfigure}{0.49\textwidth}
            \centering
        \begin{tikzpicture}[scale=1.5]
    
    \draw[line join=round,line cap=round,line width=.2cm, blue!15] (-2.5,1) -- (-1.75,1);
    \draw[line join=round,line cap=round,line width=.2cm, blue!15] (1.25,1) -- (2,1);
    \draw[line join=round,line cap=round,line width=.2cm, green!85] (0,1) -- (.75,1);
    \draw[line join=round,line cap=round,line width=.2cm, red!95] (-.5,1) -- (-1.25,1);
    \draw[line join=round,line cap=round,line width=.2cm, blue!15] (-2.5,-1) -- (-1.75,-1);
    \draw[line join=round,line cap=round,line width=.2cm, blue!15] (1.25,-1) -- (2,-1);
    \draw[line join=round,line cap=round,line width=.2cm, green!85] (0,-1) -- (.75,-1);
    \draw[line join=round,line cap=round,line width=.2cm, red!95] (-.5,-1) -- (-1.25,-1);

    \node[draw,ellipse, scale=.4, fill] (x1) at (-2.5,1) {};

    \node[draw,ellipse, scale=.4, fill] (xi1) at (-1.75,1) {};

    \node[draw,ellipse, scale=.4, fill] (xi11) at (-1.25,1) {};

    \node[draw,ellipse, scale=.4, fill] (xi2) at (-.5,1) {};

    \node[draw,ellipse, scale=.4, fill] (xi21) at (0,1) {};

    \node[draw,ellipse, scale=.4, fill] (xi3) at (.75,1) {};

    \node[draw,ellipse, scale=.4, fill] (xi31) at (1.25,1) {};

    \node[draw,ellipse, scale=.4, fill] (xC) at (2,1) {};

    \node[draw,ellipse, scale=.4, fill] (y1) at (-2.5,-1) {};

    \node[draw,ellipse, scale=.4, fill] (yj1) at (-1.75,-1) {};
    
    \node[draw,ellipse, scale=.4, fill] (yj11) at (-1.25,-1) {};

    \node[draw,ellipse, scale=.4, fill] (yj2) at (-.5,-1) {};

    \node[draw,ellipse, scale=.4, fill] (yj21) at (0,-1) {};

    \node[draw,ellipse, scale=.4, fill] (yj3) at (.75,-1) {};

    \node[draw,ellipse, scale=.4, fill] (yj31) at (1.25,-1) {};

    \node[draw,ellipse, scale=.4, fill] (yC) at (2,-1) {};

    \draw[line cap=round,line join=round,line width=.2cm, blue!15]  (xi1) to  (yj1);
    \draw[line cap=round,line join=round,line width=.2cm, red!95]  (xi11) to  (yj11);
    \draw[line cap=round,line join=round,line width=.2cm, red!95]  (xi2) to  (yj2);
    \draw[line cap=round,line join=round,line width=.2cm, green!85]  (xi21) to  (yj21);
    \draw[line cap=round,line join=round,line width=.2cm, green!85]  (xi3) to  (yj3);
    \draw[line cap=round,line join=round,line width=.2cm, blue!15]  (xi31) to  (yj31);
    \draw[out=35,in=145,line cap=round,line join=round,line width=.2cm, blue!15]  (x1) to  (xC);
    \draw[out=-35,in=-145, line cap=round,line join=round,line width=.2cm, blue!15]  (y1) to  (yC);

    \filldraw[black] (-2.5,1) circle (0pt) node [anchor=north]{$x_1$};

    \filldraw[black] (-1.75,1) circle (0pt) node[anchor=south]{$x_{i_1}$};

    \filldraw[black] (-1.25,1) circle (0pt) node[anchor=south]{$x_{i_1+1}$};

    \filldraw[black] (-.5,1) circle (0pt) node[anchor=south]{$x_{i_2}$};

    \filldraw[black] (0,1) circle (0pt) node[anchor=south]{$x_{i_2+1}$};

    \filldraw[black] (.75,1) circle (0pt) node[anchor=south]{$x_{i_3}$};

    \filldraw[black] (1.25,1) circle (0pt) node[anchor=south]{$x_{i_3+1}$};

    \filldraw[black] (2,1) circle (0pt) node [anchor=north]{$x_{|V(C)|}$};

    \filldraw[black] (-2.5,-1) circle (0pt) node [anchor=south]{$y_1$};

    \filldraw[black] (-1.75,-1) circle (0pt) node[anchor=north]{$y_{j_1}$};

    \filldraw[black] (-1.25,-1) circle (0pt) node[anchor=north]{$y_{j_1+1}$};

    \filldraw[black] (-.5,-1) circle (0pt) node[anchor=north]{$y_{j_2}$};

    \filldraw[black] (0,-1) circle (0pt) node[anchor=north]{$y_{j_2+1}$};

    \filldraw[black] (.75,-1) circle (0pt) node[anchor=north]{$y_{j_3}$};

    \filldraw[black] (1.25,-1) circle (0pt) node[anchor=north]{$y_{j_3+1}$};

    \filldraw[black] (2,-1) circle (0pt) node [anchor=south]{$y_{|V(C')|}$};

    \node[draw,ellipse, scale=.4, fill] (x1) at (-2.5,1) {};

    \node[draw,ellipse, scale=.4, fill] (xi1) at (-1.75,1) {};

    \node[draw,ellipse, scale=.4, fill] (xi11) at (-1.25,1) {};

    \node[draw,ellipse, scale=.4, fill] (xi2) at (-.5,1) {};

    \node[draw,ellipse, scale=.4, fill] (xi21) at (0,1) {};

    \node[draw,ellipse, scale=.4, fill] (xi3) at (.75,1) {};

    \node[draw,ellipse, scale=.4, fill] (xi31) at (1.25,1) {};

    \node[draw,ellipse, scale=.4, fill] (xC) at (2,1) {};

    \node[draw,ellipse, scale=.4, fill] (y1) at (-2.5,-1) {};

    \node[draw,ellipse, scale=.4, fill] (yj1) at (-1.75,-1) {};
    
    \node[draw,ellipse, scale=.4, fill] (yj11) at (-1.25,-1) {};

    \node[draw,ellipse, scale=.4, fill] (yj2) at (-.5,-1) {};

    \node[draw,ellipse, scale=.4, fill] (yj21) at (0,-1) {};

    \node[draw,ellipse, scale=.4, fill] (yj3) at (.75,-1) {};

    \node[draw,ellipse, scale=.4, fill] (yj31) at (1.25,-1) {};

    \node[draw,ellipse, scale=.4, fill] (yC) at (2,-1) {};

    \draw[dashed, thick]  (xi1) to  (yj1);
    
    \draw[dashed, thick]  (xi11) to  (yj11);
    
    \draw[dashed, thick]  (xi2) to  (yj2);
    
    \draw[dashed, thick]  (xi21) to  (yj21);
    
    \draw[dashed, thick]  (xi3) to  (yj3);
    
    \draw[dashed, thick]  (xi31) to  (yj31);

    \draw[out=35,in=145, thick]  (x1) to  (xC);
    
    \draw[out=-35,in=-145, thick]  (y1) to  (yC);

    \draw[thick] (-2.5,1) -- (2,1);
    \draw[thick] (-2.5,-1) -- (2,-1);

\end{tikzpicture}
\caption{Case $C\neq C'$ and $H'$ has many red edges.}
\label{fig:red}
\end{subfigure}
\hfill
\begin{subfigure}{0.49\textwidth}
            \centering
        \begin{tikzpicture}[scale=1.5]
    
    \draw[line join=round,line cap=round,line width=.2cm, blue!15] (-2.5,1) -- (-1.75,1);
    \draw[line join=round,line cap=round,line width=.2cm, blue!15] (1.25,1) -- (2,1);
    \draw[line join=round,line cap=round,line width=.2cm, green!85] (0,1) -- (.75,1);
    \draw[line join=round,line cap=round,line width=.2cm, red!95] (-.5,1) -- (-1.25,1);
    \draw[line join=round,line cap=round,line width=.2cm, blue!15] (-2.5,-1) -- (-1.75,-1);
    \draw[line join=round,line cap=round,line width=.2cm, blue!15] (1.25,-1) -- (2,-1);
    \draw[line join=round,line cap=round,line width=.2cm, red!95] (0,-1) -- (.75,-1);
    \draw[line join=round,line cap=round,line width=.2cm, green!85] (-.5,-1) -- (-1.25,-1);

    \node[draw,ellipse, scale=.4, fill] (x1) at (-2.5,1) {};
    \node[draw,ellipse, scale=.4, fill] (xi1) at (-1.75,1) {};
    \node[draw,ellipse, scale=.4, fill] (xi11) at (-1.25,1) {};
    \node[draw,ellipse, scale=.4, fill] (xi2) at (-.5,1) {};
    \node[draw,ellipse, scale=.4, fill] (xi21) at (0,1) {};
    \node[draw,ellipse, scale=.4, fill] (xi3) at (.75,1) {};
    \node[draw,ellipse, scale=.4, fill] (xi31) at (1.25,1) {};
    \node[draw,ellipse, scale=.4, fill] (xC) at (2,1) {};

    \node[draw,ellipse, scale=.4, fill] (y1) at (-2.5,-1) {};
    \node[draw,ellipse, scale=.4, fill] (yj1) at (-1.75,-1) {};
    \node[draw,ellipse, scale=.4, fill] (yj11) at (-1.25,-1) {};
    \node[draw,ellipse, scale=.4, fill] (yj2) at (-.5,-1) {};
    \node[draw,ellipse, scale=.4, fill] (yj21) at (0,-1) {};
    \node[draw,ellipse, scale=.4, fill] (yj3) at (.75,-1) {};
    \node[draw,ellipse, scale=.4, fill] (yj31) at (1.25,-1) {};
    \node[draw,ellipse, scale=.4, fill] (yC) at (2,-1) {};

    \draw[line cap=round,line join=round,line width=.2cm, blue!15]  (xi1) to  (yj31);
    \draw[line cap=round,line join=round,line width=.2cm, red!95]  (xi11) to  (yj3);
    \draw[line cap=round,line join=round,line width=.2cm, red!95]  (xi2) to  (yj21);
    \draw[line cap=round,line join=round,line width=.2cm, green!85]  (xi21) to  (yj2);
    \draw[line cap=round,line join=round,line width=.2cm, green!85]  (xi3) to  (yj11);
    \draw[line cap=round,line join=round,line width=.2cm, blue!15]  (xi31) to  (yj1);
    \draw[out=35,in=145,line cap=round,line join=round,line width=.2cm, blue!15]  (x1) to  (xC);
    \draw[out=-35,in=-145, line cap=round,line join=round,line width=.2cm, blue!15]  (y1) to  (yC);

    \filldraw[black] (-2.5,1) circle (0pt) node [anchor=north]{$x_1$};
    \filldraw[black] (-1.75,1) circle (0pt) node[anchor=south]{$x_{i_1}$};
    \filldraw[black] (-1.25,1) circle (0pt) node[anchor=south]{$x_{i_1+1}$};
    \filldraw[black] (-.5,1) circle (0pt) node[anchor=south]{$x_{i_2}$};
    \filldraw[black] (0,1) circle (0pt) node[anchor=south]{$x_{i_2+1}$};
    \filldraw[black] (.75,1) circle (0pt) node[anchor=south]{$x_{i_3}$};
    \filldraw[black] (1.25,1) circle (0pt) node[anchor=south]{$x_{i_3+1}$};
    \filldraw[black] (2,1) circle (0pt) node [anchor=north]{$x_{|V(C)|}$};

    \filldraw[black] (-2.5,-1) circle (0pt) node [anchor=south]{$y_1$};
    \filldraw[black] (-1.75,-1) circle (0pt) node[anchor=north]{$y_{j_1}$};
    \filldraw[black] (-1.25,-1) circle (0pt) node[anchor=north]{$y_{j_1+1}$};
    \filldraw[black] (-.5,-1) circle (0pt) node[anchor=north]{$y_{j_2}$};
    \filldraw[black] (0,-1) circle (0pt) node[anchor=north]{$y_{j_2+1}$};
    \filldraw[black] (.75,-1) circle (0pt) node[anchor=north]{$y_{j_3}$};
    \filldraw[black] (1.25,-1) circle (0pt) node[anchor=north]{$y_{j_3+1}$};
    \filldraw[black] (2,-1) circle (0pt) node [anchor=south]{$y_{|V(C')|}$};

    \node[draw,ellipse, scale=.4, fill] (x1) at (-2.5,1) {};
    \node[draw,ellipse, scale=.4, fill] (xi1) at (-1.75,1) {};
    \node[draw,ellipse, scale=.4, fill] (xi11) at (-1.25,1) {};
    \node[draw,ellipse, scale=.4, fill] (xi2) at (-.5,1) {};
    \node[draw,ellipse, scale=.4, fill] (xi21) at (0,1) {};
    \node[draw,ellipse, scale=.4, fill] (xi3) at (.75,1) {};
    \node[draw,ellipse, scale=.4, fill] (xi31) at (1.25,1) {};
    \node[draw,ellipse, scale=.4, fill] (xC) at (2,1) {};

    \node[draw,ellipse, scale=.4, fill] (y1) at (-2.5,-1) {};
    \node[draw,ellipse, scale=.4, fill] (yj1) at (-1.75,-1) {};
    \node[draw,ellipse, scale=.4, fill] (yj11) at (-1.25,-1) {};
    \node[draw,ellipse, scale=.4, fill] (yj2) at (-.5,-1) {};
    \node[draw,ellipse, scale=.4, fill] (yj21) at (0,-1) {};
    \node[draw,ellipse, scale=.4, fill] (yj3) at (.75,-1) {};
    \node[draw,ellipse, scale=.4, fill] (yj31) at (1.25,-1) {};
    \node[draw,ellipse, scale=.4, fill] (yC) at (2,-1) {};

    \draw[dashed, thick]  (xi1) to  (yj31);
    \draw[dashed, thick]  (xi11) to  (yj3);
    \draw[dashed, thick]  (xi2) to  (yj21);
    \draw[dashed, thick]  (xi21) to  (yj2);
    \draw[dashed, thick]  (xi3) to  (yj11);
    \draw[dashed, thick]  (xi31) to  (yj1);

    \draw[out=35,in=145, thick]  (x1) to  (xC);
    \draw[out=-35,in=-145, thick]  (y1) to  (yC);

    \draw[thick] (-2.5,1) -- (2,1);
    \draw[thick] (-2.5,-1) -- (2,-1);
\end{tikzpicture}
\caption{Case $C\neq C'$ and $H'$ has many blue edges.}
\label{fig:blue}
\end{subfigure}
        \caption{The cycles $C$ and $C'$ (full black) with three implanted $C_4$'s (dashed black) corresponding to three disconnected red edges of $H'$ (left) or three crossing blue edges of $H'$ (right). Performing these $C_4$-switches gives a $2$-factor of $G[V(C)\cup V(C')]$ (highlighted) consisting of three cycles.}
        \label{fig:three}
    \end{figure}

    Lastly, if $C\neq C'$ and the blue subgraph of $H'$ has at least $|X\cup Y|^{1+\eta}$ edges (see also \cref{fig:blue}), then \cref{lm:crossing}\COMMENT{Draw the vertices of $X$ and $Y$ on two parallel lines, with straight edges.
    Embed the vertices of $X$ in the order which follows $C$ and the vertices of $Y$ in the order which follows $C'$.} implies that there exist $1\leq i_1<i_2<i_3\leq |V(C)|$ and $1\leq j_1<j_2<j_3\leq |V(C')|$ such that $e_{C,i_1}e_{C',j_3}, e_{C,i_2}e_{C',j_2}$, and $e_{C,i_3}e_{C',j_1}$ are blue edges of $H'$.
    Then,
    \[x_{i_1}y_{j_3+1}y_{j_3+2}\dots y_{j_1}x_{i_3+1}x_{i_3+2}\dots x_{i_1}\]
    and
    \[x_{i_1+1}y_{j_3}y_{j_3-1}\dots y_{j_2+1}x_{i_2}x_{i_2-1}\dots x_{i_1+1}\]
    and
    \[x_{i_2+1}y_{j_2}y_{j_2-1}\dots y_{j_1+1}x_{i_3}x_{i_3-1}\dots x_{i_2+1}\]
    form a $2$-factor $\cC'$ of $G[V(C)\cup V(C')]$ with $|E(\cC')\triangle (E(C)\cup E(C'))|= 12$, so we are done.
\end{proof}

\subsection{Deriving Theorem \ref{thm:main}}\label{sec:main_proof}

\begin{proof}[Proof of \cref{thm:main}]
    Let $0< 1/n \ll \eps \ll \eta \ll \delta \leq 1$.
    Let $G$ be an $n$-vertex graph with $\delta(G)\geq n^{1-\eps}$ and let $k\in [n^{1-\delta}]$.
    Suppose that~$G$ contains a $2$-factor~$\cC_0$ consisting of $\ell\in[k]$ cycles $C_1, \dots, C_\ell$.
    Fix an (arbitrary) set $E_-\coloneqq\{x_{i+1}y_{i+1}\in E(C_{i+1}), z_iw_i\in E(C_i): i\in [\ell-1]\}$ of distinct (not necessarily disjoint) edges and let $E_+\coloneqq\{z_ix_{i+1},w_iy_{i+1}: i\in [\ell-1]\}$ and $V_E\coloneqq\{x_{i+1},y_{i+1},z_i,w_i: i\in [\ell-1]\}$.
    Then, $\cC\coloneqq \cC_0\setminus E_- \cup E_+$ is a Hamilton cycle of $G'\coloneqq G\cup E_+$.

    Denote by $E_0$ the set of edges of $\cC$ which are incident to some vertex of $V_E$ and observe that $|E_0|\leq 8(\ell-1)\leq n^{1-\eta}$.
    By \cref{cor:stage}, $G'$ contains a Hamilton cycle $\cC'$ such that $E_0\subseteq \cC'$ and $e(H_{\cC'})\geq n^{2-\eta}$.
    Then, $\cC''\coloneqq \cC'\setminus E_+\cup E_-$ is a $2$-factor of $G$ consisting of at most $\ell$ cycles.\COMMENT{To see this, observe that $\cC''$ is obtained from $\cC'$ by performing $\ell-1$ $C_4$-switches.}
    Moreover, $e(H_{\cC''})\geq n^{2-\eta}-2(\ell-1) n\geq n^{2-2\eta}$.
    Applying \cref{lm:switch} with $\cC''$ and $2\eta$ playing the roles of $\cC$ and~$\eta$ gives a $2$-factor of $G$ consisting of exactly $k$ cycles, as desired.
\end{proof}

\section{Concluding remarks}

\Cref{thm:main} provides a sufficient minimum-degree condition for a Hamiltonian graph 
to contain a \mbox{$2$-factor} consisting of precisely $k$ cycles.
The required minimum degree is a polynomial factor away from linear; this makes substantial progress towards determining the correct order of magnitude of the optimal sufficient minimum-degree condition, a problem raised in all previous papers on this topic~\cite{FGJLS05,Sar08,DFM14,BJPS20}.

With this question in mind, it is natural to wonder how close our minimum-degree condition is to being best possible.
An example of \citet[Proposition~4.1]{BJPS20} shows that, for all integers $k\geq2$ and $n\geq2k$, there are $n$-vertex Hamiltonian graphs of minimum degree $k+1$ which do not contain a $2$-factor consisting of exactly~$k$ cycles.
To our knowledge, this provides the best known lower bound on what the optimal minimum-degree condition should be for all $k\geq3$ (for $k=2$, \citet{FGJLS05} provided an example with minimum degree $4$).
In view of this construction, if~$\delta$ is bounded away from~$1$, polynomially large degrees are necessary to ensure the existence of a $2$-factor with~$n^{1-\delta}$ cycles in a Hamiltonian graph, and hence our result is best possible up to the value of~$\eps$, which we made no effort to optimise (one can easily verify from our proofs that the dependence of~$\eps$ on~$\delta$ is polynomial\COMMENT{$\eps<\delta^6/11520$ should suffice}).
It would thus be interesting to determine the best possible value of~$\eps$, as a function of~$\delta$, in \cref{thm:main}.

While our result is best possible up to the value of $\eps$ for any fixed $\delta<1$, it is unclear whether this is the case when $\delta=1-o(1)$.
As already pointed out by other authors~\cite{FGJLS05,DFM14,BJPS20}, it is conceivable that constant degrees suffice to ensure the existence of a $2$-factor consisting of a fixed number of cycles.
Some evidence for this was given by \citet{Pf04}, who showed that every claw-free Hamiltonian graph of maximum degree at least $7$ has a $2$-factor consisting of exactly two cycles.
In general, we have no reason to believe that the polynomial lower bound on the degrees in \cref{thm:main} is optimal for sub-polynomial~$k$.
Indeed, our bound on the minimum degree is an artifact of our proof, which, for one, relies heavily on ensuring that~$G$ contains many cycles of length~$4$ (an additional constraint is given by the use of \cref{lem:thomassen_new}).
One possible avenue for improving our results is to replace the use of \mbox{$C_4$-switches} by \mbox{$C_\ell$-switches}, for larger even~$\ell$, as the density required for a graph to contain many copies of such cycles decreases as $\ell$ grows.
This however increases the complexity of the arguments, so we would be interested in seeing a suitable generalisation of our approach.

The minimum-degree condition for a graph $G$ to contain a $2$-factor consisting of exactly $k$ cycles provided in \cref{thm:main} is sufficient not only when $G$ is Hamiltonian, but also under the weaker assumption that $G$ contains a $2$-factor with at most $k$ components.
Essentially, this says that having a $2$-factor with any number of cycles makes it ``much easier'' to have a $2$-factor with more cycles (from the perspective of a minimum-degree requirement).
It may be natural to consider the opposite direction: if we know that $G$ contains a $2$-factor with~$k$ cycles, what minimum degree ensures that it also contains a $2$-factor with exactly $\ell<k$ cycles?
The behaviour of this problem is completely different from the previous one. 
Indeed, consider a graph $G$ on $n\coloneqq 3(k-1)+2m$ vertices which consists of $k-1$ vertex-disjoint triangles, a copy of a complete bipartite graph $K_{m,m}$ with independent sets $A$ and $B$ which is vertex-disjoint from all the triangles, and an additional edge from every vertex in a triangle to each vertex in~$A$.
It is clear that $\delta(G)=m$ and that it contains a $2$-factor consisting of precisely $k$ cycles.\COMMENT{This is given by the $k-1$ triangles and a Hamilton cycle of the balanced complete bipartite graph.}
However, this graph cannot contain any $2$-factor with fewer cycles, for each such $2$-factor would need to have at least one cycle which, between two vertices of~$A$, only has some vertices contained in a triangle.
Such a cycle immediately causes an imbalance between~$A$ and~$B$ which cannot be rebalanced, since each vertex of~$B$ is only joined to~$A$.
The conclusion is that, for any $k$ (even allowed to grow with $n$, as long as $k=o(n)$) and any $\eps>0$, if $n$ is sufficiently large, there exist $n$-vertex graphs $G$ with $\delta(G)\geq(1/2-\eps)n$ which contain a $2$-factor consisting of exactly $k$ cycles yet they do not contain any $2$-factor with fewer than~$k$ components.

A possible avenue for further research could aim to strengthen the conclusion of our main result.
In this direction, recall that \citet{BCFGL97} showed that $\delta(G)\geq n/2$ suffices for~$G$ to contain a $2$-factor consisting of precisely $k$~cycles (for sufficiently large~$n$).
A way to strengthen the conclusion of their result is to ask not only that the $2$-factor consists of~$k$ cycles, but to prescribe their lengths.
In this direction, a well-known conjecture of \citet{ElZ84} affirms that, in order for a graph~$G$ on $n=\sum_{i=1}^kn_i$ vertices to contain a $2$-factor with cycles of lengths $n_1,\ldots,n_k$, it suffices that $\delta(G)\geq\sum_{i=1}^k\lceil n_i/2\rceil$, and this condition would be best possible.
This conjecture is known to hold for several particular cases.
As examples, it is known to hold in all cases when $k=2$~\cite{ElZ84}; 
the well-known Corr\'adi-Hajnal theorem~\cite{CH63} corresponds to the case that $n=3k$ and $n_i=3$ for all $i\in[k]$, and the case that $n=4k$ and all cycles have length $4$ was settled by \citet{Wang10}.\COMMENT{(see~\cite{CY18} for a survey of such results)}
In a similar direction as in El-Zahar's conjecture, one could ask to prescribe the lengths of the cycles in \cref{thm:main}.
Thus, we propose the following general problem.

\begin{problem}
    Let $n_1,\ldots,n_k\geq3$ be positive integers, and $n\coloneqq\sum_{i=1}^kn_i$.
    What is the optimal condition on the minimum degree of an $n$-vertex Hamiltonian graph under which we can we guarantee that it contains a $2$-factor consisting of cycles of lengths $n_1,\ldots,n_k$?
\end{problem}

As a more particular question, given the sublinear minimum-degree condition in \cref{thm:main}, one could be interested in understanding for which families of values of $n_1,\ldots,n_k$ (if any) the optimal minimum degree condition grows sublinearly with~$n$.
This is certainly not the case in general.
One immediately notices that $\delta(G)>n/2$ becomes a necessary condition if any of the desired lengths are odd (indeed, a balanced complete bipartite graph is Hamiltonian and does not contain any odd cycles), but there are also many sets of even lengths for which the required minimum degree remains linear, and in some cases one cannot even improve on the minimum-degree condition predicted by \citet{ElZ84}. 
\COMMENT{We show this for some values of $n$, assuming certain divisibility conditions, but the examples we propose can be generalised in different ways.}

For a first example of this last behaviour, let $k\geq3$ be an integer and consider a graph on $n=4k+2$ vertices consisting of two vertex-disjoint cliques of order $2k+1$ which are joined by a matching of size~$2$.
It is clear that this graph is Hamiltonian and has minimum degree $2k=n/2-1$.
However, it does not contain a $2$-factor consisting of two cycles of lengths $2k+2$ and~$2k$ (indeed, any cycle which is not entirely contained in one of the two cliques must contain at least two vertices in the other).
Thus, for these choices of $n$ and lengths, the Hamiltonicity assumption does not allow us to improve the minimum-degree condition of \citet{ElZ84}.
A slight modification of this example when $n=4k$ is not divisible by~$8$ and both cliques have size~$2k$ shows that, for the case that all cycles have length~$4$, the Hamiltonicity assumption does not allow us to improve on the minimum-degree condition proved by \citet{Wang10}.\COMMENT{Indeed, let $G$ be a graph on $n=4k$ vertices, where $k$ is odd, which we obtain as follows.
First, consider the vertex-disjoint union of two cliques of order $2k$.
Fix four distinct vertices $u,v$ in one of the cliques and $x,y$ in the other, and add the edges $ux$ and $vy$ to the graph, and remove the edge $uv$.
Since $k$ is odd, neither of the two cliques can contain a $C_4$-factor, so in order to build one at least one cycle must live ``across'' both cliques.
But such a cycle would need to be $uxyv$, which is not contained in $G$.}

A simple modification of the example above provides an almost-tight condition for the case that we want an even number of cycles, all of the same even length $\ell$.\COMMENT{For the particular case that $\ell=4$, this together with the previous example covers all possible values of $n$.}
Indeed, for any given positive integers $h$ and $\ell\geq 3$, let $n\coloneqq2h\ell$ and consider a graph which consists of the vertex-disjoint union of two cliques of sizes $h\ell-1$ and $h\ell+1$, plus a matching of size $2$ joining them.
We note that this graph is Hamiltonian and has minimum degree $n/2-2$ but does not contain a $2$-factor consisting of~$2h$ cycles all of length~$\ell$.
Indeed, if we consider the smaller clique, it can contain at most $h-1$ cycles of length~$\ell$, and there is at most one cycle which contains vertices from both cliques.
However, such a cycle cannot have length~$\ell$ if it contains $\ell-1$ vertices in the smaller clique, which would be needed so that all vertices in the clique are contained in some cycle.
Thus, the sufficient minimum degree for a Hamiltonian graph on $n$ vertices to contain a $2$-factor with all cycles of length~$\ell$ must be at least $n/2-1$.
This example can easily be modified to deal with other families of (not all equal) lengths to obtain other almost tight lower bounds.


As a further observation, we note that, for any collection of even lengths $n_1\leq\ldots\leq n_k$, in order for a graph $G$ on $n\coloneqq\sum_{i=1}^kn_i$ vertices to contain a $2$-factor with cycles of such lengths, a sufficient minimum-degree condition must satisfy that $\delta(G)=\Omega(\max\{n_1,n/n_k\})$.
Indeed, we first give a simple example with $\delta(G)=\Theta(n_1)$ which does not contain the desired $2$-factor.
Let $G$ be a graph whose vertex set is equipartitioned into $\lceil n/(n_1-1)\rceil$ sets (so they all have size at most $n_1-1$), where each part forms a clique and the parts are joined by disjoint edges in a cyclic manner.
It is clear that~$G$ is Hamiltonian and that $\delta(G)=\Theta(n_1)$.
Clearly, none of the cliques may contain a cycle of length $n_1$, so a cycle of length $n_1$ must ``go around'' the different cliques.
However, at most one cycle can go around, so each other cycle would need to be contained in one of the cliques, which are too small to contain them.
The second example, which shows that $\delta(G)\geq n/(n_k/2+1)-2$ does not suffice, is fairly similar.
Consider a graph $G$ which consists of $N\coloneqq n_k/2+1$ cliques, all of size roughly $n/N$ but slightly unbalanced so that at least one of the cliques has odd order,\COMMENT{Note that this is only needed if $n/N$ is an even integer, otherwise any equipartition works.} joined by disjoint edges in a cyclic manner.
It is clear that $\delta(G)\geq n/N-2$ and that it is Hamiltonian.
Any clique of odd order may not be decomposed into cycles of even lengths, so at least one cycle would need to ``go around'' all the cliques.
But any such cycle must have length at least $2N>n_k$, a contradiction.
For most families of lengths, these lower bounds provide a clear separation between this problem and the one tackled in \cref{thm:main}, where we do not have any such lower bounds for the optimal minimum-degree condition.

The examples discussed above rely on divisibility constraints.
One can remove these constraints by not prescribing all the lengths of the cycles, and instead prescribing only some of them, leaving some flexibility for how the $2$-factor is completed.
In this direction, we propose the following problem.

\begin{problem}
    Is it true that for every $\eps>0$ there exists some $\delta>0$ such that the following holds?\\
    \indent    
    Let $n,n_1,\ldots,n_k$ be positive even integers such that $n_i\leq\delta n$ for all $i\in[k]$ and $\sum_{i=1}^kn_i\leq(1-\eps)n$.
    If~$G$ is an \mbox{$n$-vertex} Hamiltonian graph with $\delta(G)\geq\eps n$, then $G$ contains a $2$-factor which contains cycles of lengths $n_1,\ldots,n_k$.
\end{problem}

\bibliographystyle{mystyle} 
\bibliography{refs}

\end{document}